\documentclass[11pt]{article}
\usepackage{amsmath, amssymb, amsthm, graphicx}
\usepackage[makeroom]{cancel}
\usepackage{tcolorbox}
\usepackage[colorlinks=true,linkcolor=blue,citecolor=blue]{hyperref}
\usepackage{tikz-cd}
\usepackage{geometry}
\usepackage{mathrsfs}  
\usepackage{subfig}
\usepackage{ifthen}

\newcommand{\N}{\mathbb{N}}
\newcommand{\W}{\mathcal{W}}
\newcommand{\Z}{\mathbb{Z}}

\newcommand{\Nei}{\mathcal{N}}
\newcommand{\dom}{\mathrm{dom}}

\newcommand{\diam}{\mathop{\mathrm{diam}}}
\newcommand{\Aut}{\mathrm{Aut}}

\newcommand{\floor}[1]{\lfloor #1 \rfloor}
\newcommand{\ceil}[1]{\lceil #1 \rceil}

\renewcommand{\int}[1]{\mathrm{int(#1)}}

\newcommand{\im}[1]{\mathrm{im}(#1)}

\newtheorem{theorem}{Theorem}[section]
\newtheorem{lemma}{Lemma}[section]

\newtheorem{proposition}{Proposition}[section]

\theoremstyle{definition}
\newtheorem{definition}{Definition}[section]

\theoremstyle{remark}

\newtheorem{remark}{Remark}[section]

\newcommand{\Zb}{\overline{\Z}}
\newcommand{\smalltile}[1]{\raisebox{-1mm}{\includegraphics{Tiles/smalltile#1}}}
\newcommand{\pattern}[1]{\begin{smallmatrix}#1\end{smallmatrix}}
\newcommand{\whitetile}{\raisebox{-.5mm}{\includegraphics{Tiles/tiles-7}}}
\newcommand{\blacktile}{\raisebox{-.5mm}{\includegraphics{Tiles/tiles-8}}}

\newcommand{\gen}{\mathscr{L}^0}
\newcommand{\nar}{\mathscr{L}^1}

\definecolor{mypurple}{rgb}{1, .5, 1}
\definecolor{myred}{rgb}{1, .5, .5}
\definecolor{myblue}{rgb}{.5, .5, 1}
\definecolor{mygreen}{rgb}{.5, 1, .5}

\definecolor{darkpurple}{rgb}{.8, .4, .8}
\definecolor{darkred}{rgb}{.8, .4, .4}
\definecolor{darkblue}{rgb}{.4, .4, .8}
\definecolor{darkgreen}{rgb}{.4, .8, .4}

\newcommand{\purple}{\textcolor{darkpurple}{purple}}
\newcommand{\red}{\textcolor{darkred}{red}}

\newcommand{\green}{\textcolor{darkgreen}{green}}

\title{Local generation of tilings: the even bicolor Wang tilesets}
\author{Juliette Favereau\footnote{\'Ecole des Mines de Nancy} ~and Mathieu Hoyrup\footnote{Universit\'e de Lorraine, CNRS, Inria, LORIA, Nancy, 54000, France}}
\date{November 15, 2024}

\begin{document}
\maketitle

\begin{abstract}
In this article, we apply the techniques developed in our previous article ``Local generation of tilings'', in which we introduced two definitions capturing the intuitive idea that some subshifts admit a procedure that can generate any tiling and working in a local way. We classify all the Wang tilesets with two colors in which each tile has an even number of each color.

\end{abstract}

\tableofcontents

\ifdefined\addparentheses
\newcommand{\ti}[1]{(#1)}
\else
\newcommand{\ti}[1]{#1}
\fi

\newcommand{\myref}[1]{%
    \ifthenelse{\equal{#1}{ex_geom}}{3.2}{%
    \ifthenelse{\equal{#1}{def_equivalence}}{3.2}{%
    \ifthenelse{\equal{#1}{thm_obs_C0}}{5.1}{%
    \ifthenelse{\equal{#1}{rmk_obs_sub}}{5.1}{%
    \ifthenelse{\equal{#1}{thm_notin_nar}}{6.1}{%
    \ifthenelse{\equal{#1}{prop_weak_factor_mixing}}{3.4}{%
    \ifthenelse{\equal{#1}{ex_block}}{3.3}{%
    \ifthenelse{\equal{#1}{sec_corners}}{7.3.2}{%
    \ifthenelse{\equal{#1}{sec_all_even}}{7.3.1}{%
    \ifthenelse{\equal{#1}{prop_inclusion}}{2.4}{%
    \ifthenelse{\equal{#1}{prop_separation}}{4.10}{%
    \ifthenelse{\equal{#1}{def_equivalence}}{3.2}{%
    \ifthenelse{\equal{#1}{thm_block_gen}}{7.1}{%
    }}}}}}}}}}}}}}
\section{Introduction}
In \cite{FH24}, we have introduced two classes of subshifts,~$\gen\subseteq\nar$, that are intended to capture the intuitive idea that certain subshifts admit a procedure that can generate all their tilings in a local way.

In the present article, we classify all the Wang tilesets consisting of bicolor tiles having an even number of each color, according to these classes. Although there are 255 such non-empty tilesets, their symmetries reduce that number to 36 equivalence classes. We find that 13 of them belong to~$\gen$, 8 do not belong to~$\gen$ and 7 do not belong to~$\nar$ (the other 8 equivalence classes of tilesets contain a tile that cannot be used in a tiling, therefore they induce the same subshift as a smaller tileset). We conjecture that the 8 classes of tilesets that do not belong to~$\gen$ do not belong to~$\nar$ either.

The article is organized as follows. In Section \ref{sec_def}, we recall the notions of local generations introduced in \cite{FH24}. Section \ref{sec_obstructions} gathers the tilesets that do not belong to~$\gen$. The tilesets that do not belong to~$\nar$ are presented in Section \ref{sec_transitions}. In Section \ref{sec_positive}, we show that all the other tilesets belong to~$\gen$. Section \ref{sec_classification} summarizes the results by listing all the equivalence classes of even bicolor Wang tilesets, and their classification w.r.t.~the classes $\gen$ and~$\nar$.

\section{Background}\label{sec_def}

Let us first recall the classical definitions needed for this article.

\subsection{Group actions}
For~$d\geq 1$,~$(\Z^d,+,0)$ is an abelian group. A~$\Z^d$\textbf{-action} on a set~$E$ is a homomorphism from~$\Z^d$ to the group of bijections from~$E$ to~$E$ with the composition operation. We also say that~$E$ is a~$\Z^d$\textbf{-set}. The result of applying the image of~$p\in\Z^d$ under the homomorphism to~$e\in E$ is written as~$p\cdot e\in E$. A \textbf{continuous~$\Z^d$-action} on a topological space~$X$ is a homomorphism from~$\Z^d$ to the group of homeomorphisms from~$X$ to~$X$ with the composition operation. We also say that~$X$ is a~$\Z^d$\textbf{-space}. A~$\Z^d$-space~$Y$ is a \textbf{factor} of a~$\Z^d$-space~$X$ if there exists a continuous surjective map~$f:X\to Y$ that commutes with the actions:~$p\cdot f(x)=f(p\cdot x)$. $f$ is called a \textbf{factor map}. If~$f$ is a homeomorphism, then~$f$ is a \textbf{conjugacy} and~$X$ and~$Y$ are \textbf{conjugate}.

A~$\Z^e$-space~$Y$ is a \textbf{weak factor} of a~$\Z^d$-space~$X$ if there exists a continuous surjective map~$f:X\to Y$ and a group homomorphism~$\varphi:\Z^e\to\Z^d$ such that~$p\cdot f(x)=f(\varphi(p)\cdot x)$. If~$f$ is a homeomorphism and~$\varphi$ an isomorphism, then~$X$ and~$Y$ are \textbf{weakly conjugate} (see \cite[Definition \myref{def_equivalence}]{FH24}).

\subsection{Symbolic dynamics}\label{sec_dynamics}

If~$A$ is a finite alphabet and~$E$ is countable, then~$A^E$ is endowed with the Cantor topology, which is the product of the discrete topology on~$A$. This makes~$A^E$ a compact metrizable space. The \textbf{shift action} on~$A^{\Z^d}$ is the continuous~$\Z^d$-action defined, for~$p\in\Z^d$ and~$x\in A^{\Z^d}$, by~$p\cdot x=y$ where~$y(q)=x(p+q)$. A~\textbf{$\Z^d$-subshift} is a compact set~$X\subseteq A^{\Z^d}$ which is shift-invariant, i.e.~satisfies~$\sigma^p(X)=X$ for all~$p\in\Z^d$. A~\textbf{$\Z^d$-fullshift} is~$A^{\Z^d}$ for some finite alphabet~$A$. If~$H$ is a subgroup of~$\Z^d$, then a configuration~$x\in A^{\Z^d}$ is~$H$\textbf{-periodic} if~$x(p+h)=x(p)$ for all~$p\in\Z^d$ and~$h\in H$. The~\textbf{$H$-periodic shift} over~$\Sigma$ is the subshift~$X_H\subseteq\Sigma^{\Z^d}$ defined as the set of all the~$H$-periodic configurations.

Let~$d\geq 1$ and~$\Sigma$ a finite alphabet. We endow~$\Z^d$ with the metric~$d(p,q)=\max_i|p_i-q_i|$, where~$p=(p_1,\ldots,p_d)$ and~$q=(q_1,\ldots,q_d)$. For~$F\subseteq\Z^d$,~$\diam(F)=\max_{p,q\in F}d(p,q)$, which is~$\infty$ if~$F$ is infinite. For~$F,G\subseteq\Z^d$,~$d(F,G)=\min_{p\in F,q\in G}d(p,q)$. We will often consider the \textbf{cubes}~$S_n=[0,n-1]^d$ and~$Q_n=[-n,n]^d$, for~$n\in\N$.

Let~$F\subseteq\Z^d$ be any set. An~$F$\textbf{-pattern} is an element~$\pi\in \Sigma^F$. A \textbf{pattern} is an~$F$-pattern for some~$F$, which is called the \textbf{domain} of~$\pi$ and is denoted by~$\dom(\pi)$. A pattern~$\pi'$ \textbf{extends} a pattern~$\pi$ if~$\dom(\pi)\subseteq\dom(\pi')$ and~$\pi'(p)=\pi(p)$ for all~$p\in\dom(\pi)$. A \textbf{finite pattern} is an~$F$-pattern for some finite~$F\subseteq\Z^d$. If~$\pi,\pi'$ are patterns with disjoint domains~$F,F'$ respectively, then~$\pi\cup\pi'$ is the pattern with domain~$F\cup F'$ extending~$\pi$ and~$\pi'$. A \textbf{configuration} is an element~$x\in\Sigma^{\Z^d}$.

If~$\pi$ is a pattern, then~$[\pi]$ is the set of configurations extending~$\pi$. The set~$\Sigma^{\Z^d}$ is endowed with the topology generated by the sets~$[\pi]$ where~$\pi$ is a finite pattern. As already mentioned, the space~$\Sigma^{\Z^d}$ is endowed with the \textbf{shift action}, which is the continuous~$\Z^d$-action~$p\cdot x=y$ where~$y(q)=x(p+q)$. We will often write~$\sigma^p(x):=p\cdot x$. The shift also acts on patterns: if~$\pi$ is an~$F$-pattern, then~$\sigma^p(\pi)$ is the~$F'$-pattern~$\pi'$ where~$F'=F-p$ and~$\pi'(q)=\pi(p+q)$ for~$q\in F'$. If~$p\in\Z^d$ and~$x$ is a configuration, then we say that a pattern~$\pi$ \textbf{appears at position~$p$ in~$x$} if~$\sigma^p(x)$ extends~$\pi$.

A \textbf{subshift} is a closed subset~$X$ of~$\Sigma^{\Z^d}$ which is \textbf{shift-invariant}, i.e.~satisfies~$\sigma^p(X)=X$ for all~$p\in\Z^d$. We fix a subshift~$X\subseteq \Sigma^{\Z^d}$. All the subsequent notions are relative to~$X$, but we do not mention~$X$ which will always be clear from the context. A \textbf{valid configuration} is an element~$x\in X$. A \textbf{valid pattern} is a pattern appearing in some valid configuration, at any position or equivalently at the origin. Two disjoint regions~$F,G\subseteq\Z^d$ are \textbf{independent} if for every valid~$F$-pattern~$\pi$ and every valid~$G$-pattern~$\pi'$,~$\pi\cup \pi'$ is valid.

A \textbf{subshift of finite type (SFT)} is the subshift induced by a finite set~$P$ of finite patterns, called forbidden patterns, and defined as the set of configurations in which no~$\pi\in P$ appears. Let~$C$ be a finite set of colors. A \textbf{Wang tile} over~$C$ associates to each edge of the unit square a color in~$C$. A \textbf{Wang tileset} over~$C$ is a set~$T$ of Wang tiles over~$C$. It induces a~$\Z^2$-subshift of finite type~$X_T\subseteq T^{\Z^2}$ which is the set of configurations in which neighbor cells have the same color on their common edge.

\subsection{Dynamical properties}

We recall classical dynamical properties of subshifts, that can be found in \cite{Furstenberg67} or \cite{Vries13}.
\begin{definition}
A subshift~$X\subseteq\Sigma^{\Z^d}$ is \textbf{transitive} if for every pair of valid finite patterns~$\pi,\pi'$, there exists~$p\in\Z^d$ such that~$\pi\cup \sigma^p(\pi')$ is a valid pattern; in other words,~$\pi$ and~$\pi'$ appear in a common configuration. 

A subshift~$X\subseteq\Sigma^{\Z^d}$ is \textbf{strongly irreducible} if there exists~$n\in\N$ such that all regions~$F,G\subseteq \Z^d$ satisfying~$d(F,G)\geq n$ are independent.

A subshift~$X\subseteq\Sigma^{\Z^d}$ is \textbf{mixing} if for every~$m\in\N$ there exists~$n\in\N$ such that all regions~$F,G\subseteq \Z^d$ satisfying~$\diam(F)\leq m,\diam(G)\leq m$ and~$d(F,G)\geq n$ are independent.

A subshift~$X\subseteq\Sigma^{\Z^d}$ is \textbf{weakly mixing} if~$X\times X$ is transitive, equivalently if for every~$n\in\N$, there exists~$p\in\Z^d$ such that~$Q_n$ and~$p+Q_n$ are independent.
\end{definition}
 The equivalent definition of weak mixing was proved in \cite{FH24}.

One has the following chain of implications:
\begin{center}
strongly irreducible $\implies$ mixing $\implies$ weakly mixing $\implies$ transitive.
\end{center}

\subsection{Local generation}
We recall the notions of local generation introduced in \cite{FH24}.
\begin{definition}[\ti{The class $\gen$}]
Let~$d\geq 1$. We define~$\gen_d$ as the smallest class of subshifts containing the~$\Z^d$-fullshifts, the periodic~$\Z^d$-shifts and the countable~$\Z^d$-subshifts, and which is closed under finite products and factors. Let~$\gen=\bigsqcup_{d\geq 1}\gen_d$.
\end{definition}

Let~$A,B$ be finite alphabets,~$E,F$ be countable sets and~$f:A^E\to B^F$ be continuous w.r.t.~the product topologies. To each~$q\in F$ is associated~$\W_f(q)\subseteq E$, which is the minimal region on which the values of any~$x\in A^E$ determine the value of~$f(x)$ at~$q$.
\begin{definition}[\ti{Narrow function}]
Let~$r\in\N$. A continuous function~$f:A^E\to B^F$ is \textbf{$r$-narrow} if for all~$q\in F$,~$|\W_f(q)|\leq r$. We say that~$f$ is \textbf{narrow} if it is~$r$-narrow for some~$r\in\N$.
\end{definition}

Note that a function is~$0$-narrow if and only if it is constant. The identity~$f:A^E\to A^E$ is~$1$-narrow. The composition of two narrow functions is narrow. 

We come to the second notion of local generation. Let~$\Sigma$ be finite and~$F$ be countable.

\begin{definition}[\ti{The class $\nar$}]
A compact set~$X\subseteq\Sigma^F$ belongs to~$\nar$ if it is a countable union of images of narrow functions, i.e.~$X=\bigcup_{n\in\N} X_n$ where~$X_n=\im{f_n}$ and~$f_n:A_n^\N\to X$ is~$r_n$-narrow.
\end{definition}

The second notion of local generation is a relaxation of the first one \cite[Propositions \myref{prop_inclusion} and \myref{prop_separation}]{FH24}.
\begin{proposition}
One has~$\gen\subsetneq \nar$.
\end{proposition}

\subsection{Symmetry group}\label{sec_group}
We are going to analyze bicolor Wang tilesets. There are $2^4=16$ possible tiles, so~$2^{16}-1=65535$ possible tilesets (excluding the empty tileset). We can reduce the number of cases by observing that there is a group of order 32 acting on bicolor tilesets, such that two tilesets that are equivalent under this group action induce weakly conjugate subshifts, so we only need to consider one representative for each equivalence class. A computer program shows that there are $2890$ equivalence classes. This number is still large for a systematic classification, so in this article we only consider the \textbf{even bicolor tilesets}, which only use tiles having an even number of edges of each color. There are~$8$ even tiles, which we draw using wires instead of colors for aesthetic reasons: \smalltile{0000}, \smalltile{1100}, \smalltile{0110}, \smalltile{0011}, \smalltile{1001}, \smalltile{1010}, \smalltile{0101}, \smalltile{1111}. There~$2^8-1=255$ even non-empty tilesets, and a computation shows that there are~$36$ equivalence classes under the symmetry group action.

Let us describe the symmetry group. It contains:
\begin{itemize}
\item Eight geometric transformations:
\begin{itemize}
\item Rotations by 0, 90, 180 or 270 degrees,
\item Reflections accross the horizontal, vertical or one of the diagonal axes,
\end{itemize}
\item Chromatic transformations:
\begin{itemize}
\item Complementation of the colors on the horizontal edges, or on the vertical edges, or both,
\end{itemize}
\item Combinaison of a geometric and a chromatic transformation.
\end{itemize}

The geometric transformations are elements of the dihedral group~$D_4$, which is the symmetry group of the square. The chromatic transformations are elements of the Klein group~$K_4=\Z/2\Z\times \Z/2\Z$. These two types of transformations do not commute, but~$K_4$ is a normal subgroup of~$G$ so~$G$ is a semidirect product~$K_4\rtimes_\phi D_4$ for some homomorphism~$\phi:D_4\to \Aut(K_4)$, therefore~$G$ has order~$|D_4|\times |K_4|=8\times 4=32$. The homomorphism~$\phi$ sends the rotation by~$90$ degrees to the permutation between horizontal and vertical color complementations, and sends reflections to the identity. It corresponds to the fact that a rotation followed by a horizontal color complementation equals a vertical color complementation followed by the rotation, and that a reflection and a color complementation commute.

\begin{proposition}
For every bicolor tileset~$T$ and every~$g\in G$,~$T$ and~$g\cdot T$ induce weakly conjugate subshifts.
\end{proposition}
\begin{proof}
Each geometric~$g$ transformation induces a geometric equivalence as in Example \myref{ex_geom} in \cite{FH24}. Each chromatic transformation induces a conjugacy. These two types of transformations generate the group and weak conjugacy between subshifts is transitive, so every group element induces a weak conjugacy between subshifts.
\end{proof}

\section{Subshifts that do not belong to \texorpdfstring{$\gen$}{L0}}\label{sec_obstructions}
We first recall the main technique developed in \cite{FH24} to show that a two-dimensional subshift does not belong to~$\gen$, and then apply it to the even bicolor tilesets.

\subsection{Ramified subshift}
We recall the following notions introduced in \cite{FH24}, where they are discussed in more details.

If~$F\subseteq\Z^2$ and~$r\in\N$, then let~$\Nei(F,r)=\{p:\exists q\in F,d(p,q)\leq r\}$ be the~$r$-neighborhood of~$F$.
\begin{definition}[\ti{Graft}]
Let~$X$ be a~$\Z^2$-subshift,~$F\subseteq\Z^2$,~$\pi$ an~$F$-pattern,~$x\in X$ and~$r\in\N$. We say that~$\pi$ can be \textbf{$r$-grafted into~$x$ at position~$p\in\Z^2$} if there exists~$y\in X$ that coincides with~$x$ outside~$\Nei(p+F,r)$ and such that~$\pi$ appears at position~$p$ in~$y$, i.e.~$y(p+q)=\pi(q)$ for~$q\in F$.
\end{definition}

\begin{definition}[\ti{Ramification}]
Let~$X$ be a~$\Z^2$-subshift. A \textbf{ramification} is a configuration~$x\in X$ together with~$r\in\N,u,v\in\Z^2$ and~$F\subseteq\Z^2$ such that for all~$\lambda,\mu\in\Z$ with~$\mu>0$, the~$F$-pattern appearing at position~$\mu u+\lambda v$ in~$x$ cannot be~$r$-grafted into~$x$ at position~$\lambda v$. We say that~$x$ is an~$(r,v)$\textbf{-ramification}.
\end{definition}

Say that two vectors~$v_1,v_2\in\Z^2$ are \textbf{independent} if~$\lambda_ 1v_1+\lambda_2v_2=0$ implies~$\lambda_1=\lambda_2=0$, for~$\lambda_1,\lambda _2\in\Z$.
\begin{definition}[\ti{Ramified subshift}]
A~$\Z^2$-subshift~$X$ is \textbf{ramified} if for every~$r\in\N$ there exist infinitely many pairwise independent vectors~$v$ such that~$X$ admits a ramification of radius~$r$ and support~$v$.
\end{definition}
If~$r<r'$, then an~$(r',v)$-ramification is also an~$(r,v)$-ramification so, in order to show that~$X$ is ramified, it is sufficient to build a sequence~$(v_r)_{r\in\N}$ of pairwise independent vectors such that~$X$ admits an~$(r,v_r)$-ramification for each~$r$.

\begin{remark}[\ti{Subshift of a ramified subshift}]\label{rmk_obs_sub}
As observed in \cite[Remark \myref{rmk_obs_sub}]{FH24} if~$T$ is a Wang tileset and~$X_T$ has ramifications that can all be built using a proper subset of tiles~$R\subsetneq T$, then for each tileset~$S$ such that~$R\subseteq S\subseteq T$,~$X_S$ is ramified as well.
\end{remark}

If a subshift is ramified, then it is not strongly irreducible. Actually, showing the lack of strong irreducibility will always be the first step when building ramifications.

The main technique to prove that a subshift is not in~$\gen$ is Theorem \myref{thm_obs_C0} in \cite{FH24}.
\begin{theorem}[\ti{An obstruction to being in $\gen$}]\label{thm_obs_C0}
Let~$X$ be a~$\Z^2$-subshift. If~$X$ is ramified and weakly mixing, then~$X\notin \gen$.
\end{theorem}

\subsection{Applications}
There are exactly $8$ equivalence classes of even bicolor tilesets to which Theorem \ref{thm_obs_C0} can be applied.

\begin{theorem}\label{thm_wires_notC0}
The subshifts induced by the following tilesets are ramified and mixing, therefore do not belong to~$\gen$:
\begin{align*}
T_1&=\raisebox{-1mm}{\includegraphics{even_classes/small_class-14-1}} & \text{(Class \ref{class_stretched_irr_stairs})}\\
T_2&=\raisebox{-1mm}{\includegraphics{even_classes/small_class-23-4}} & \text{(Class \ref{class_irr_stairs})}\\
T_3&=\raisebox{-1mm}{\includegraphics{even_classes/small_class-15-1}} & \text{(Class \ref{class_flames})}\\
T_4&=\raisebox{-1mm}{\includegraphics{even_classes/small_class-27-0}} & \text{(Class \ref{class_corners_white})}\\
T_5&=\raisebox{-1mm}{\includegraphics{even_classes/small_class-34-1}} & \text{(Class \ref{class34})}\\
T_6&=\raisebox{-1mm}{\includegraphics{even_classes/small_class-25-4}} & \text{(Class \ref{class25})}\\
T_7&=\raisebox{-1mm}{\includegraphics{even_classes/small_class-29-1}} & \text{(Class \ref{class_wires})}\\
T_8&=\raisebox{-1mm}{\includegraphics{even_classes/small_class-35-1}} & \text{(Class \ref{class35})}
\end{align*}
\end{theorem}

We do not know whether these subshifts belong to~$\nar$ and leave it as an open question (see Section \ref{sec_transitions} for further discussion).

We do not need a specific argument for each one of the 8 tilesets listed above, but an argument for the 2 tilesets~$T_4$ and~$T_7$ containing the corners \smalltile{1100}, \smalltile{0110}, \smalltile{0011} and \smalltile{1001}, and an argument for the other 6.

\subsection{First case: all the corners} We prove Theorem \ref{thm_wires_notC0} for the following tilesets:
\begin{align*}
T_4&= \raisebox{-1mm}{\includegraphics{Wires/wires-3}} &\text{(Class \ref{class_corners_white})}\\
T_7&= \raisebox{-1mm}{\includegraphics{Wires/wires-1}} &\text{(Class \ref{class_wires})}
\end{align*}
We start by showing that the subshift~$X_{T_7}$ is ramified.
 
First,~$X_{T_7}$ is not strongly irreducible: a diagonal~$D_r:=\{(x,x):0\leq x\leq r\}$ filled with the tile \smalltile{1001} entirely determines the lower triangle~$\{(x,y):0\leq x\leq r,0\leq y\leq x\}$ and in particular the cell~$C_r:=(r,0)$ (see Figure \ref{fig_corners_2_even_not_SI}). Therefore,~$D_r$ and~$C_r$ are not independent and~$d(D_r,C_r)=\ceil{\frac{r}{2}}$ can be arbitrarily large.

\begin{figure}[!ht]
\centering
\includegraphics{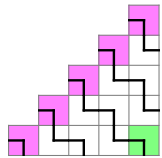}
\caption[Not strongly irreducible]{$X_{T_7}$ is not strongly irreducible: the \purple\ and \green\ regions are not independent, and can be taken arbitrarily far from each other.}\label{fig_corners_2_even_not_SI}
\end{figure}

We then create ramifications illustrated in Figure \ref{fig_corners_white_obs}, showing that~$X_{T_7}$ is ramified. It is clear from the picture that~$(r,v)$-ramifications can be built for any~$r$ and any~$v=(n,1)$ if~$n$ is sufficiently large. We assume that the picture is self-explanatory and does not need a formal argument.

These ramifications do not use the tile \smalltile{1111}, i.e.~belong to~$X_{T_4}$ which is therefore ramified as well by Remark \ref{rmk_obs_sub}.

\begin{figure}[!ht]
\centering
\includegraphics{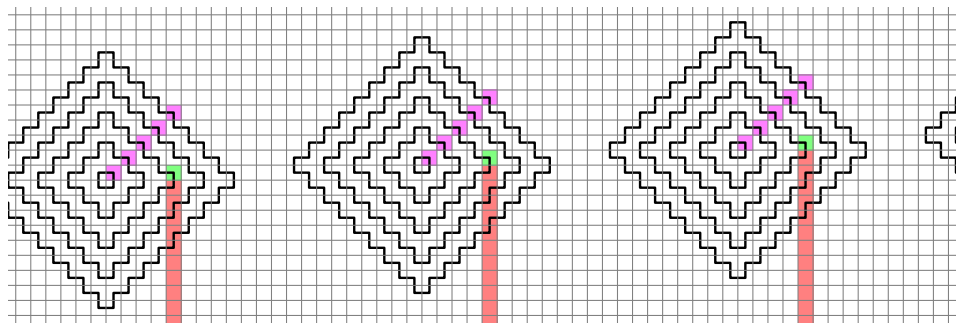}
\caption{A ramification in~$X_{T_4}$ and~$X_{T_7}$: the content of a \red\ cell cannot be grafted in a \green\ cell without changing the content of the corresponding \purple\ region. The distance~$r$ between the \purple\ and \green\ regions can be made arbitrary large by increasing the number of stairs. The vector~$v$ between two consecutive \green\ cells can take value~$(n,1)$ for any sufficiently large~$n$. One has~$u=(0,-1)$ and~$F=\{(0,0)\}$}\label{fig_corners_white_obs}
\end{figure}

We now prove that both subshifts are mixing, by showing that any finite pattern can be extended to a finite pattern with white boundary.

\begin{lemma}\label{lemma_extending_pi}
For both~$T_4$ and~$T_7$, any valid pattern on~$[0,n-1]^2$ can be extended to a valid pattern on~$[-2n,2n-1]^2$ whose boundary is completely white.
\end{lemma}
\begin{proof}
Let us first show the result for~$X_{T_4}$. Let~$\pi$ be a valid~$[0,n-1]^2$-pattern. We first extend~$\pi$ to a valid~$[0,2n-1]^2$-pattern whose right and upper boundaries are white.

Let~$u_0\in\{0,1\}^n$ be the upper boundary of~$\pi$. We inductively define~$u_i,\ldots,u_n\in\{0,1\}^n$ as follows. Let~$i<n$ and assume that~$u_i$ has been defined. Decompose~$u_i$ as a concatenation of the strings~$11,01$ and~$0$, with possibly a single string~$1$ at the beginning, and note that the decomposition is unique because these strings are not suffixes of each other, so starting from the end of~$u_i$, the choice is alway unique. Then define~$u_{i+1}$ from~$u_i$ by replacing~$11$ with~$00$,~$01$ with~$10$,~$0$ with~$0$, and the possible~$1$ at the beginning with~$0$. Several observations can be made: the string~$11$ can appear only in~$u_0$; the position of the rightmost~$1$ in~$u_i$ is strictly increasing in~$i$, so~$u_n=0^n$.

We can define a~$[0,n-1]\times [n,2n-1]$-pattern whose lower boundary is~$u_0$, such that row~$n+i$ has lower boundary~$u_i$ and upper boundary~$u_{i+1}$ (in particular the upper boundary of the pattern is~$u_n$, i.e.~it is white), and whose right boundary is white (see Figure \ref{fig_row_ui}).
\begin{figure}[!ht]
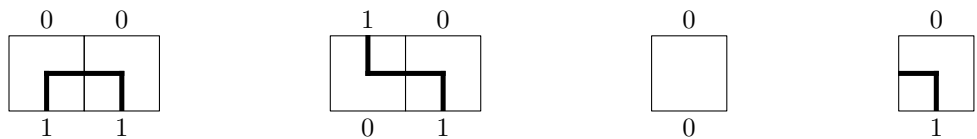

\centering
\hspace{\stretch{1}}\includegraphics{Mixing/pattern-1}\hspace{\stretch{2}}\includegraphics{Mixing/pattern-2}\hspace{\stretch{2}}\includegraphics{Mixing/pattern-3}\hspace{\stretch{2}}\includegraphics{Mixing/pattern-4}\hspace{\stretch{1}}
\caption[Mixing]{Filling the row between~$u_i$ and~$u_{i+1}$.}\label{fig_row_ui}
\end{figure}

In the same way, we can define a~$[n,2n-1]\times [0,n-1]$-pattern whose left-boundary is the right-boundary of~$\pi$ and whose right and upper boundaries are white (apply the same technique up to a symmetry accross the diagonal).

Then fill the rectangle~$[n,2n-1]^2$ with white tiles. We obtain a~$[0,2n-1]^2$-pattern extending~$\pi$, whose right and upper boundaries are white. We then copy symmetric versions of this pattern to fill~$[-2n,2n-1]^2$ (see Figure \ref{fig_extending_pi}), which is possible because applying a symmetry to each tile yields a tile that belongs to the tileset.
\begin{figure}[!ht]
\centering
\includegraphics{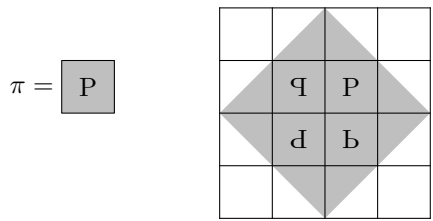}
\caption[Mixing]{Extending~$\pi$.}\label{fig_extending_pi}
\end{figure}

Observe that the same construction can be applied to~$X_{T_7}$, because~$T_7$ contains~$T_4$ and has the same symmetries.
\end{proof}

Finally, the following argument applies to both~$X_{T_4}$ and~$X_{T_7}$. For~$m\in\N$, let~$n=4m+1$. Let~$F,G\subseteq\Z^2$ be squares of side length~$m$ satisfying~$d(F,G)\geq n$. If~$\pi,\pi'$ are valid patterns on~$F$ and~$G$ respectively, then by Lemma \ref{lemma_extending_pi} they can be extended on the larger squares~$\Nei(F,2m)$ and~$\Nei(G,2m)$ to valid patterns with white boundaries. These larger squares are disjoint, so one can fill the rest of the plane with white cells, yielding a valid configuration and shows that~$F$ and~$G$ are independent. Therefore, both tilesets are mixing.

\subsection{Second case: missing corners}
We now prove Theorem \ref{thm_wires_notC0} for the other tilesets:
\begin{align*}
T_1&=\raisebox{-1mm}{\includegraphics{even_classes/small_class-14-1}}&\text{(Class \ref{class_stretched_irr_stairs})}\\
T_2&=\raisebox{-1mm}{\includegraphics{even_classes/small_class-23-4}}&\text{(Class \ref{class_irr_stairs})}\\
T_3&=\raisebox{-1mm}{\includegraphics{even_classes/small_class-15-1}}&\text{(Class \ref{class_flames})}\\
T_5&=\raisebox{-1mm}{\includegraphics{even_classes/small_class-34-1}}&\text{(Class \ref{class34})}\\
T_6&=\raisebox{-1mm}{\includegraphics{even_classes/small_class-25-4}}&\text{(Class \ref{class25})}\\
T_8&=\raisebox{-1mm}{\includegraphics{even_classes/small_class-35-1}}&\text{(Class \ref{class35})}
\end{align*}
Note that all these tilesets contain~$T_1$ and are contained in~$T_8$.

Figure \ref{fig_irr_stairs_not_SI} illustrates that~$X_{T_8}$ is not strongly irreducible. The only tile having a black left edge and a white upper edge is \smalltile{1010}, therefore a diagonal made of this tile forces the lower triangle to be filled with the same tile.
\begin{figure}[!ht]
\centering
\includegraphics{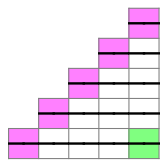}
\caption[Not strongly irreducible]{The subshifts are not strongly irreducible: the \purple{} diagonal and the \green{} cell are not independent.}\label{fig_irr_stairs_not_SI}
\end{figure}

This lack of strong irreducibility can be turned into ramifications, illustrated in Figure \ref{fig_irr_stairs_obstruction}. These ramifications only use the tiles from~$T_1$, so all the tilesets are ramified by Remark \ref{rmk_obs_sub}.
\begin{figure}[!ht]
\centering
\includegraphics{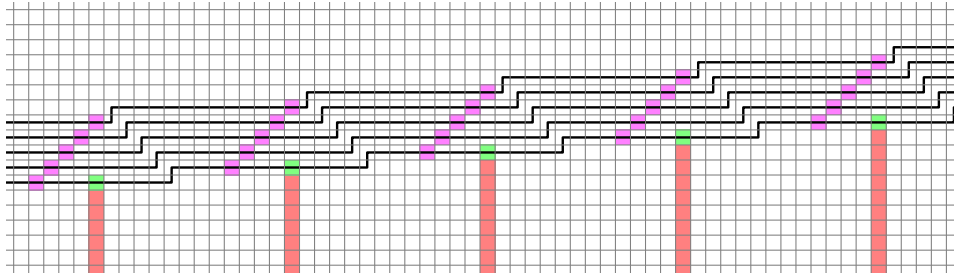}
\caption[ramification]{A ramification: the content of a \red\ cell cannot be grafted in a \green\ cell without changing the content of the corresponding \purple\ region. The distance~$r$ between the \purple\ and \green\ regions can be made arbitrary large by increasing the number of stairs. The vector~$v$ between two consecutive \green\ cells can take value~$(n,1)$ for any sufficiently large~$n$. One has~$u=(0,-1)$ and~$F=\{(0,0)\}$}\label{fig_irr_stairs_obstruction}
\end{figure}


We finally show that all these tilesets induce mixing subshifts. We need two different arguments, depending on whether the tilesets contain \smalltile{0101}, presented as Lemmas \ref{lem_stairs_mixing} and \ref{lem_stairs_mixing1} below.

\begin{lemma}\label{lem_stairs_mixing}
The tilesets $T_1=$~\raisebox{-1mm}{\includegraphics{even_classes/small_class-14-1}} and $T_3=$~\raisebox{-1mm}{\includegraphics{even_classes/small_class-15-1}} induce mixing subshifts.
\end{lemma}
\begin{proof}
Let~$S_n=[0,n-1]^2$. We show that if~$|p|\geq 7n$, then~$S_n$ and~$p+S_n$ are independent. Let~$p=(i,j)$, there are two cases: either~$|j|\geq 3n$ or~$|j|<3n$. The argument is illustrated in Figure \ref{fig_irr_stairs_mixing}.

\begin{figure}[!ht]
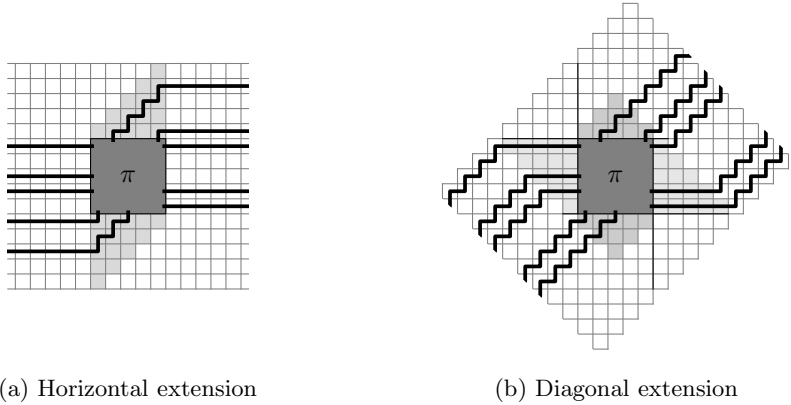

\centering
\hspace{\stretch{1}}
\subfloat[Horizontal extension]{\includegraphics{Widgets/widget-3}\label{fig_hori_extension}}
\hspace{\stretch{1}}
\subfloat[Diagonal extension]{\includegraphics{Widgets/widget-4}\label{fig_diag_extension}}
\hspace{\stretch{1}}
\caption[Mixing]{Extending a pattern~$\pi$ from a square to an infinite strip with white boundaries}\label{fig_irr_stairs_mixing}
\end{figure}

First assume that~$|j|\geq 3n$, i.e.~$p+S_n$ is vertically far from~$S_n$. Any valid~$S_n$-pattern~$\pi$ can be extended to a valid pattern on the horizontal strip~$\Z\times [-n,2n-1]$ whose upper and lower boundaries are white (see Figure \ref{fig_hori_extension}). An important point is that the two tilesets considered here do not allow two adjacent horizontal edges to be black, because there is no way of placing matching tiles above them. Therefore, the upper and lower boundaries of~$\pi$ cannot contain two consecutive black edges, which enables one to make stairs grow out of these boundaries. Therefore, any valid pattern~$\pi'$ on~$p+S_n$ can be extended to a valid pattern on~$\Z\times [j-n,j+2n-1]$ in a similar way. As~$|j|\geq 3n$, these strips are disjoint and we can fill the other cells with the white tile, yielding a valid configuration extending~$\pi\cup\pi'$, so~$S_n$ and~$p+S_n$ are independent.

Now assume that~$|j|<3n$, which implies~$|i-j|\geq |i|-|j|>4n$. Any valid~$S_n$-pattern~$\pi$ can be extended to a valid pattern on the diagonal strip~$D=\{(k,l):|k-l|\leq 2n\}$ with white boundaries (see Figure \ref{fig_diag_extension}). Therefore, any valid pattern~$\pi'$ on~$p+S_n$ can be extended to a valid pattern on~$p+D=\{(k,l):|k-i-l+j|\leq 2n\}$. The strips~$D$ and~$p+D$ are disjoint, because~$4n<|i-j|\leq |(k-i)-(l-j)|+|k-l|$, so~$(k,l)$ cannot belong both to~$D$ and~$p+D$. Again, one can fill all the other cells with the white tile, yielding a valid extension of~$\pi\cup\pi'$, so~$S_n$ and~$p+S_n$ are independent.

Note that the extensions do not use the tile \smalltile{0110}, so they apply to both tilesets.
\end{proof}

We now give an argument for the tilesets containing \smalltile{0101}.
\begin{lemma}\label{lem_stairs_mixing1}
If a bicolor tileset contains~$T_2=$~\raisebox{-1mm}{\includegraphics{even_classes/small_class-23-4}}, then it induces a mixing subshift.
\end{lemma}
\begin{proof}
Let again~$S_n=[0,n-1]^2$. We show that if~$|p|\geq 3n$, then~$S_n$ and~$S_n$ are independent. Let~$p=(i,j)$, there are two cases:~$|i|\geq 3n$ or~$|j|\geq 3n$. The argument is illustrated in Figure \ref{fig_irr_stairs_mixing1}.

\begin{figure}[!ht]
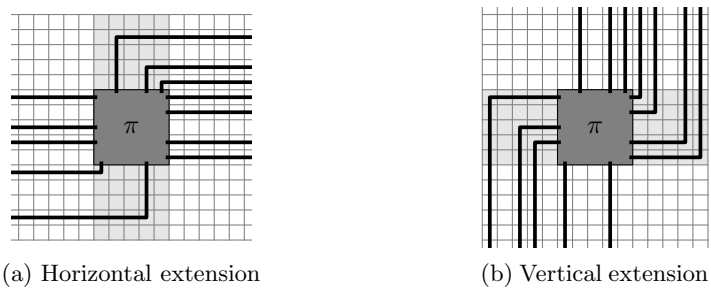

\centering
\hspace{\stretch{1}}
\subfloat[Horizontal extension]{\includegraphics{Widgets/widget-1}\label{fig_hori_extension_bis}}
\hspace{\stretch{1}}
\subfloat[Vertical extension]{\includegraphics{Widgets/widget-2}\label{fig_vert_extension}}
\hspace{\stretch{1}}
\caption[Mixing]{Extending a pattern $\pi$ from a square to an infinite strip with white boundaries}\label{fig_irr_stairs_mixing1}
\end{figure}

First assume that~$|j|\geq 3n$, i.e.~that~$p+S_n$ is vertically far from~$S_n$. A valid~$S_n$-pattern~$\pi$ can be first extended using~$T_2$ to a valid pattern on~$[0,n-1]\times [-n,2n-1]$-pattern whose upper and lower boundaries are white, and then extended using~$T_2$ to a valid pattern on the horizontal strip~$\Z\times [-n,2n-1]$ whose upper and lower boundaries are white (see Figure \ref{fig_hori_extension_bis}. Therefore,~$S_n$ and~$p+S_n$ are independent because valid patterns on these squares can be extended to valid patterns on two disjoint horizontal strips, which can be jointly extended by using the white tile on the other cells.

Now assume that~$|i|\geq 3n$, i.e.~$p+S_n$ is horizontally far from~$S_n$. A symmetric construction can be made to extend a valid~$S_n$-pattern to a valid pattern on the vertical strip~$[-n,2n-1]\times \Z$ with white left and right boundaries (consider a symmetry across the diagonal and note that it preserves~$T_2$, see Figure \ref{fig_vert_extension}), so the same argument shows that~$S_n$ and~$p+S_n$ are independent.
\end{proof}

\section{Subshifts that do not belong to \texorpdfstring{$\nar$}{nar}}\label{sec_transitions}

The main technique to prove that a subshift does not belong to~$\nar$ is Theorem \myref{thm_notin_nar} from \cite{FH24}. It involves~$\Zb$, which is the one-dimensional SFT on the alphabet~$\{\raisebox{-.5mm}{\includegraphics{Tiles/tiles-7}},\raisebox{-.5mm}{\includegraphics{Tiles/tiles-8}}\}$ with forbidden pattern \raisebox{-.5mm}{\includegraphics{Tiles/tiles-6}}.

\begin{theorem}[An obstruction to being in $\nar$]\label{thm_notin_nar}
Let~$X$ be a~$\Z^d$-subshift. If~$X$ is weakly mixing and~$\Zb$ is a weak factor of~$X$, then~$X\notin \nar$.
\end{theorem}

Among the $36$ equivalence classes of even bicolor tilesets, there are exactly $6$ classes to which Theorem \ref{thm_notin_nar} can be applied, and therefore do not belong to~$\nar$. There is a 7th even bicolor tileset to which Theorem \ref{thm_notin_nar} cannot be applied because it is not weakly mixing, but which is not in~$\nar$ by a similar argument (see below).

In Theorem \ref{thm_wires_notC0} we identified tilesets that do not belong to~$\gen$. We expect that they do not belong to~$\nar$ either, but have not been able to prove it so far. The technique used in this section, namely Theorem \ref{thm_notin_nar}, cannot be applied to them because they are all mixing and Proposition \myref{prop_weak_factor_mixing} in \cite{FH24} prevents~$\Zb$ from being a weak factor of a mixing subshift, because~$\Zb$ is neither periodic nor mixing.

\begin{theorem}\label{thm_tilesets_nar}
The subshifts induced by the following tilesets are weakly mixing and weak factor to~$\Zb$; therefore, they do not belong to~$\nar$:
\begin{align*}
T_1&=\raisebox{-1mm}{\includegraphics{even_classes/small_class-19-4}}&\text{(Class \ref{class_irr_chevron})}\\
T_2&=\raisebox{-1mm}{\includegraphics{even_classes/small_class-9-1}}&\text{(Class \ref{class_histograms})}\\
T_3&=\raisebox{-1mm}{\includegraphics{even_classes/small_class-32-1}}&\text{(Class \ref{class_crossing_brackets})}\\
T_4&=\raisebox{-1mm}{\includegraphics{even_classes/small_class-21-8}}&\text{(Class \ref{class_staples_gaps})}\\
T_5&=\raisebox{-1mm}{\includegraphics{even_classes/small_class-24-1}}&\text{(Class \ref{class_slanted_histograms})}\\
T_6&=\raisebox{-1mm}{\includegraphics{even_classes/small_class-33-2}}&\text{(Class \ref{class_overlapping_staples})}
\end{align*}

The subshift induced by the following tileset does not belong to~$\nar$:
\begin{align*}
T_7&=\raisebox{-1mm}{\includegraphics{even_classes/small_class-20-6}}\hspace{5.5mm}&\hspace{7.5mm}\text{(Class \ref{class_nested_staples})}
\end{align*}
\end{theorem}

The proof is separated in several statements.
\begin{lemma}\label{lem_weak_factor_zbar}
The subshifts induced by tilesets $T_1$ to~$T_7$ weak factor to~$\Zb$.
\end{lemma}
\begin{proof}
Let~$T$ be one of those tilesets except~$T_5$. Let~$f:X_T\to \{\whitetile,\blacktile\}^\Z$ extract the first row of a configuration and replace the tiles \smalltile{0000} and \smalltile{0101} by \whitetile\ and the other tiles by \blacktile\ (note that~$T_2$ and~$T_7$ do not contain \smalltile{0101}, in which case only \smalltile{0000} is replaced by \whitetile). Let~$\varphi:\Z\to\Z^2$ be the homomorphism sending~$p$ to~$(p,0)$. One has~$\sigma^p\circ f=f\circ \sigma^{\varphi(p)}$, so~$f$ is a weak factor map. Its image is contained in~$X_{\Zb}$ because the only tiles in~$T$ that can be put at the left of \smalltile{0000} or \smalltile{0101} is one of these two tiles, so in the output of~$f$, the only tile that can be appear on the left of the white tile is the white tile. Finally,~$f$ is surjective because there exists a configuration in~$X_T$ whose first row is:
\begin{center}
\includegraphics{Widgets/widget-5}\hspace{1cm}or\hspace{1cm}\includegraphics{Widgets/widget-8}
\end{center}
and the image of such a configuration has a dense orbit in~$X_{\Zb}$ (the first picture applies to any~$T$ except~$T_7$, to which the second picture applies).

Let now~$T=T_5$. Let~$f:X_T\to  \{\whitetile,\blacktile\}^\Z$ extract the diagonal and replace any non-white tile by \blacktile. Let~$\varphi:\Z\to\Z^2$ send~$p$ to~$(p,p)$. One has~$\sigma^p\circ f=f\circ \sigma^{\varphi(p)}$ so~$f$ is a weak factor map. Its image is contained in~$X_{\Zb}$, because one easily checks that in a configuration~$x\in X_T$, if the white appears at position~$(p,p)$ then it also appears at position~$(p-1,p-1)$, so~$f$ does not produce the forbidden pattern. Finally~$f$ is surjective because~$X_T$ contains the following configuration:
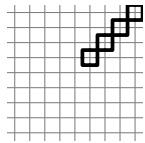
\begin{center}
\includegraphics{Widgets/config-0}
\end{center}
\end{proof}

We now prove that~$T_1,T_3,T_4$ and~$T_6$ induce weakly mixing subshifts, because they contain~$T_1$.
\begin{lemma}
Any bicolor tileset that contains~$T_1=$ \raisebox{-1mm}{\includegraphics{even_classes/small_class-19-4}} induces a weakly mixing subshift.
\end{lemma}
\begin{proof}
Any valid~$S_n$-pattern can be extended to a valid pattern on~$[-n,+\infty)^2\setminus [n,+\infty)^2$ whose boundary is white, as illustrated in Figure \ref{fig_chevron_extension}.
\begin{figure}[!ht]
\centering
\includegraphics{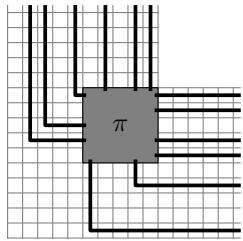}
\caption[Weak mixing]{Extending a pattern~$\pi$ from~$S_n$ to~$[-n,+\infty)^2\setminus [n,+\infty)^2$.}\label{fig_chevron_extension}
\end{figure}
Let~$\pi$ and~$\pi'$ be valid patterns on~$S_n$ and~$(2n,2n)+S_n$ respectively. They can be extended to valid pattern~$\xi$ and~$\xi'$ on the disjoint regions~$[-n,+\infty)^2\setminus [n,+\infty)^2$ and~$[n,+\infty)^2\setminus [3m,+\infty)^2$ respectively, which have white boundaries. The rest of the plane can be filled with white tiles, yielding a valid configuration extending~$\pi$ and~$\pi'$. Therefore,~$S_n$ and~$(2n,2n)+S_n$ are independent.
\end{proof}

We briefly explain why~$T_2$ and~$T_5$ induce weakly mixing subshifts.

For $T_2=$ \raisebox{-1mm}{\includegraphics{even_classes/small_class-9-1}}, we claim that the lower half-plane~$A:=\Z\times (-\infty,0]$ and the upper half-plane~$B:=\Z\times [3,+\infty)$ are independent. Any valid pattern on either of these regions has a boundary which contains at most one black edge. It can be extended to the next row so that the extended pattern has a white boundary. The two extended patterns match and cover~$\Z^2$, yielding a common valid extension of the original patterns. It implies that~$S_m$ and~$(0,m+2)+S_m$ are independent, because they are contained in~$(0,m-1)+A$ and~$(0,m-1)+B$ respectively.

%
%

For $T_5=$ \raisebox{-1mm}{\includegraphics{even_classes/small_class-24-1}}, we claim that the lower-right half-plane~$A:=\{(i,j):j\leq i\}$ and the upper-left half-plane~$B:=\{(i,j):j\geq i+5\}$ are independent. Any valid pattern on~$A$ can be extended to a valid pattern on~$\{(i,j):j\leq i+2\}$ with a white boundary, using only the tiles \smalltile{0000}, \smalltile{0011} and \smalltile{0110}. By symmetry of the tileset and of~$A$ and~$B$ across the diagonal, any valid pattern on~$B$ can be extended to a valid pattern on~$\{(i,j):j\geq i+3\}$ with a white boundary. These two patterns match and cover~$\Z^2$, yielding a common valid extension of the original patterns. It implies that~$S_m$ and~$(0,2m+3)+S_m$ are independent, because they are contained in~$(0,m-1)+A$ and~$(0,m-1)+B$ respectively.

The subshift induced by~$T_7=$ \raisebox{-1mm}{\includegraphics{even_classes/small_class-20-6}} is not weakly mixing. We do not prove it because we do not need it, but the problem is that in a valid configuration, either all the occurrences of \smalltile{0110} appear in even cells and all the occurrences of \smalltile{0011} appear odd cells, or the converse holds (the parity of a cell~$(i,j)$ is~$i+j\mod 2$). Therefore, two different cells are never independent.

However,~$X_{T_7}$ weak factors to a~$\Z^2$-subshift~$Y$ which is weakly mixing, and weak factors to~$\Zb$ as well.

We could prove the result directly with the subshift~$X$ induced by~$T_7$, but it is much easier to understand the argument on the following subshift~$Z$ which is conjugate to~$X$. $Z$ is a subshift on~$\{\whitetile,\blacktile\}$ which has the white configuration, the black configuration and for each sequence~$(n_j)_{j\in\Z}$ satisfying~$|n_{j+1}-n_j|=1$, the configuration whose cell~$(i,j)$ is white if~$i<n_j$ and black otherwise. In other words,~$Z$ is the two-dimensional SFT on~$\{0,1\}$ with forbidden patterns $\pattern{1&0}$, $\pattern{1&1\\0&0}$, $\pattern{0&0\\1&1}$ and $\pattern{0&1\\0&1}$, and a typical configuration is a zigzag separating a black region on the right from a white region on the left. The conjugacy map~$F:Z\to X$ intuitively replaces each black square by a black boundary and shifts the configuration by~$(\frac{1}{2},\frac{1}{2})$. More formally, the rule underlying~$F$ is shown in Figure \ref{fig_rule_T7}.
\begin{figure}[!ht]
\centering
\includegraphics{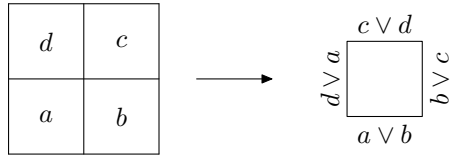}
\caption[Local rule]{A tile in~$T_7$ is determined by the 4 input bits assigned to its corners}\label{fig_rule_T7}
\end{figure}
We let the reader check that this rule is correct, i.e.~only outputs tiles in~$T_7$.

As~$X_{T_7}$,~$Z$ is not weakly mixing. However, it weak factors to a weakly mixing~$\Z^2$-subshift~$Z_0$which itself weak factors to~$\Zb$.

Let~$Z_0$ be obtained as~$Z$, but using sequence~$(n_j)_{j\in\Z}$ satisfying the relaxed condition~$|n_{j+1}-n_j|\leq 1$. The weak factor map~$:Z\to Z_0$ converts a sequence~$(n_j)_{j\in\Z}$ to the sequence~$n'_j=\floor{\frac{n_{2j}}{2}}$, and the homomorphism is~$\varphi(p)=2p$.

As in Lemma \ref{lem_weak_factor_zbar} the function~$g:Z_0\to\Zb$ that extracts the first row is a weak factor map via~$\varphi(p)=(p,0)$.

$Z_0$ is weakly mixing because~$S_m$ and~$(0,2m-1)+S_m$ are independent. If~$\pi$ and~$\xi$ are valid patterns on these two regions and~$a\in [0,m]$ is the first positions of a black tile on row~$m-1$ in~$\pi$ ($a=m$ if there is no black tile) and~$b\in [0,m]$ is the first position of a black tile on row~$2m-1$ of~$\xi$ ($b=m$ if there is no black tile), then we can define a sequence~$n_i$ for~$m-1\leq i\leq 2m-1$ with~$n_{m-1}=a$,~$n{2m-1}=b$ and~$|n_{i+1}-n_i|\leq 1$ for~$m-1\leq i<2m-1$, simply because~$|b-a|\leq m=(2m-1)-(m-1)$. From this sequence we can build a valid extension of~$\pi\cup\xi$.

\section{Subshifts that belong to \texorpdfstring{$\gen$}{L0}}\label{sec_positive}

We now present the even bicolor tilesets that belong to~$\gen$. We only give details for the more difficult cases.  The simpler ones will be briefly described in Section \ref{sec_classification}. For instance, we will see that all the even bicolor tilesets of size at most 3 induce subshifts that belong to~$\gen$ (when they are non-empty).

We first recall Theorem \myref{thm_block_gen} from \cite{FH24}, which will help establishing the results.
\begin{theorem}[\ti{$\gen$ and higher power presentations}]\label{thm_block_gen}
Let~$a\in\Z^d$ have positive coordinates and~$X$ be a~$\Z^d$-subshift. One has
\begin{equation*}
X\in\gen\iff [X]_a\in\gen.
\end{equation*}
\end{theorem}

\subsection{Corners and 2 even tiles}\label{sec_corners_2_even_II}
There are two equivalence classes of tilesets consisting of the 4 corners and 2 even tiles. We saw one of them (Class \ref{class_wires}) in Theorem \ref{thm_wires_notC0} where we showed that the induced subshift is not in~$\gen$. On the contrary, the other class induces a subshift in~$\gen$. A representative is shown in Figure \ref{fig_corners_2_even} (Class \ref{class_corners_2_even_II}).
\begin{figure}[!ht]
\centering
\includegraphics{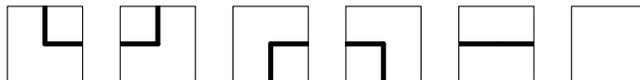}
\caption{The corners and 2 even tiles}\label{fig_corners_2_even}
\end{figure}

Let~$T$ be the complete even bicolor tileset and let~$F:\{0,1\}^{\Z^2}\to X_T$ be the factor map from Section \myref{sec_all_even} in \cite{FH24} (assign the input bits to the points of the grid, and if bits~$a$ and~$b$ are are assigned to the endpoints of an edge, then assign color~$a+b\mod 2$ to that edge). Let~$S\subseteq T$ be the tileset shown in Figure \ref{fig_corners_2_even}. Observe that it is obtained from~$T$ by removing the tiles whose horizontal edges are both black. Therefore,~$X_S$ is the image by~$F$ of the SFT~$X\subseteq \{0,1\}^{\Z^2}$ with 4 forbidden patterns~$\begin{smallmatrix}a&b\\c&d\end{smallmatrix}$ with~$a\neq b$ and~$c\neq d$.

We show that~$X\in\gen$, which implies that~$X_S\in\gen$ as~$X_S$ is a factor of~$X$. We define a continuous function~$G:\{0,1\}^{\Z^2}\to X$, which leaves every valid configuration unchanged, and corrects the invalid ones ($G$ is a retraction). The function~$G$ applies local rules that are periodic with periods~$(4,0)$ and~$(0,5)$, so the~$(4,5)$-higher power presentation of~$X$ is a factor of the~$(4,5)$-higher power presentation of the fullshift, which is a fullshift, implying that~$X\in\gen$ (Theorem \ref{thm_block_gen}).

Let~$R=[0,4]\times [0,5]$ and~$\partial R=(\{0,4\}\times [0,5])\cup ([0,4]\times \{0,5\})$ be its boundary. The function~$G$ works in two steps, i.e.,~$G$ is a composition~$G=G_2\circ G_1$:
\begin{enumerate}
\item First step:~$G_1$ replaces, at every position~$p=(4x,5y)$, a possible pattern~$\begin{smallmatrix}b&&&&c\\\text{\boxed{a}}&\overline{a}&a&\overline{a}&a\\d&&&&e\end{smallmatrix}$ with~$b\neq c$ or~$d\neq e$ by the pattern~$\begin{smallmatrix}b&&&&c\\\text{\boxed{a}}&a&a&a&a\\d&&&&e\end{smallmatrix}$ (the box indicates the origin position in the pattern, and~$\overline{a}=1-a$),
\item Second step: $G_2$ replaces, at every position~$p=(4x,5y)$, a possible~$R$-pattern~$\pi$ containing a forbidden pattern, by a valid pattern that coincides with~$\pi$ on its boundary~$\partial R$.
\end{enumerate}

The next result implies that~$G$ does not affect the valid configurations.
\begin{lemma}
The patterns~$\begin{smallmatrix}b&&&&c\\\text{\boxed{a}}&\overline{a}&a&\overline{a}&a\\d&&&&e\end{smallmatrix}$ with~$b\neq c$ or~$d\neq e$ are invalid. 
\end{lemma}
\begin{proof}
Assume that~$c\neq b$, the case~$e\neq d$ is symmetric. For any pattern~$\begin{smallmatrix}x_0&x_1&x_2&x_3&x_4\\\text{\boxed{a}}&\overline{a}&a&\overline{a}&a\end{smallmatrix}$ with~$x_0=b$ and~$x_4=c$, there exists~$i<4$ such that~$x_i\neq x_{i+1}$, so it contains a forbidden pattern~$\begin{smallmatrix}x_i&x_{i+1}\\a&\overline{a}\end{smallmatrix}$ or~$\begin{smallmatrix}x_i&x_{i+1}\\\overline{a}&a\end{smallmatrix}$.
\end{proof}

We now show that~$G_2$ is well-defined.
\begin{lemma}
Let~$\pi$ be an~$R$-pattern such that no pattern~$\begin{smallmatrix}b&&&&\overline{b}\\\text{\boxed{a}}&\overline{a}&a&\overline{a}&a\end{smallmatrix}$ appears in~$\pi$ at position~$(0,0)$ and no pattern~$\begin{smallmatrix}\text{\boxed{a}}&\overline{a}&a&\overline{a}&a\\d&&&&\overline{d}\end{smallmatrix}$ appears in~$\pi$ at position~$(0,5)$.

There exists a valid~$R$-pattern~$\pi'$ that coincides with~$\pi$ on its boundary~$\partial R$.
\end{lemma}
\begin{proof}
For~$j\in\{1,2,3\}$, let~$s_j$ and~$t_j$ be the starting symbol and ending symbol of row number~$j$ in~$\pi$. The row number~$j$ of~$\pi'$ will be made of a block of~$s_j$'s followed by a block of~$t_j$'s. We need to choose where the transition occurs (if~$s_j=t_j$, then there is no transition and the row is constant).

For~$j=1$, we choose a transition to avoid forbidden patterns in~$[0,4]\times [0,1]$: either~$s_1=t_1$, in which case row~$1$ is constant and no forbidden pattern can appear, or~$s_1\neq t_1$, in which case the assumption implies that there exists~$i<4$ such that~$\pi(i,0)=\pi(i+1,0)$. We then make the transition happen between~$i$ and~$i+1$, i.e.~the row~$1$ of~$\pi'$ is $s_i^{i+1}t_i^{4-i}$.

For~$j=3$, we apply a symmetric argument to avoid forbidden patterns in~$[0,4]\times [4,5]$.

We finally fill row~$j=2$. If~$s_2=t_2$ then again, row~$2$ is constant. If~$s_2\neq t_2$, then we need to make the transition happen at a position which is different from the possible transitions occurring in rows~$1$ and~$3$. It is always possible, because there are 3 possible positions.
\end{proof}

After applying~$G$, the~$R$-patterns appearing at positions~$(4x,5y)$ contain no forbidden pattern. In~$\Z^2$, any~$2\times 2$ square is contained in some block~$(4x,5y)+R$ because they overlap, so the configuration~$G(x)$ does not contain any forbidden pattern.

A configuration obtained by applying this local generation procedure to a random grid of bits is shown in Figure \ref{fig_random_corners_2_even}.

\begin{figure}[!ht]
\centering
\includegraphics{Positive/corners_2_even_bis-0}
\caption[Random tiling]{A random tiling using the even bicolor tiles except \smalltile{0101} and  \smalltile{1111}}\label{fig_random_corners_2_even}
\end{figure}

\subsection{Corners and 3 even tiles}\label{sec_corners_3_even}
The tileset~$S$ containing all the even tiles except the black tile \smalltile{1111} (Figure \ref{fig_corners_3_even}, Class \ref{class_corners_3_even}), induces a subshift that belongs to~$\gen$:
\begin{figure}[!ht]
\centering
\includegraphics{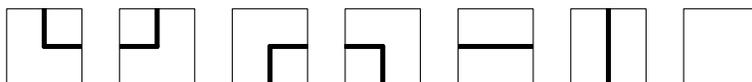}
\caption{The corners and 3 even tiles}\label{fig_corners_3_even}
\end{figure}

As in the previous section, let~$T$ be the tileset containing all the even bicolor tiles. Let~$R=[0,2]\times [0,3]$ and~$[X_S]_{(2,3)}$ the~$(2,3)$-higher power presentation of~$X_S$. It turns out that all the boundaries of the valid~$R$-patterns from~$X_T$, which are exactly the ones in which the black color appears an even number of times, are also boundaries of valid~$R$-patterns from~$X_S$. In other words, it is possible to change the content of an~$R$-pattern from~$X_T$ so that the new pattern has the same boundary, and does not contain the black tile \smalltile{1111} (see Lemma \ref{lem_correct} below). Therefore,~$[X_S]_{(2,3)}$ is a factor of~$[X_T]_{(2,3)}$, with a factor map that replaces each~$R$-pattern containing \smalltile{1111} by an~$R$-pattern with the same boundary and without the forbidden tile. Therefore, one has:
\begin{equation*}
X_T\in \gen\implies [X_T]_{(2,3)}\in\gen\implies [X_S]_{(2,3)}\in \gen\implies X_S\in \gen.
\end{equation*}

%

\begin{lemma}\label{lem_correct}
Every $R$-pattern made of even bicolor tiles has the same boundary as a~$R$-pattern made of even bicolor tiles and avoiding the black tile \smalltile{1111}.
\end{lemma}
\begin{proof}
As explained in Section \ref{sec_corners_2_even_II}, an~$R$-pattern from the even tileset can be obtained by assigning bits to the corners, which form a~$4\times 5$ grid. The tile \smalltile{1111} corresponds to the patterns $\begin{smallmatrix}0&1\\1&0\end{smallmatrix}$ and~$\begin{smallmatrix}1&0\\0&1\end{smallmatrix}$. There is a way to correct a $4\times 5$ grid filled with bits in order to remove these patterns while leaving the bits on the boundary unchanged, as illustrated in Figure \ref{fig_correction_rule}. The idea is that after correction, each $2\times 2$ square will contain $2$ horizontally or vertically adjacent cells with the same bit. The corresponding~$R$-pattern made of even tiles will then have the same boundary as the original one and will not contain the tile \smalltile{1111}.
\begin{figure}[!ht]
\centering
\includegraphics{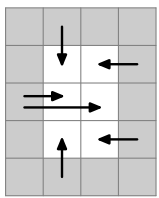}
\caption[Local rule]{Correction rule: copying the content of boundary cells to inner cells, as indicated by the arrows, removes the forbidden patterns $\begin{smallmatrix}0&1\\1&0\end{smallmatrix}$ and~$\begin{smallmatrix}1&0\\0&1\end{smallmatrix}$.}\label{fig_correction_rule}
\end{figure}
\end{proof}

%

A configuration obtained by applying the local generation procedure to a random grid of bits is shown in Figure \ref{fig_random_corners_3_even}.
\begin{figure}[!ht]
\centering
\includegraphics{Positive/even_tiling_no_black_bis-1}
\caption[Random tiling]{A random tiling with the even bicolor tiles except \smalltile{1111}}\label{fig_random_corners_3_even}
\end{figure}

\section{Classification of the even bicolor tilesets}\label{sec_classification}
We summarize the classification of the even bicolor tilesets.

We recall that there are 36 equivalence classes of such tilesets under the symmetry group action, that 8 of them are not minimal, i.e.~induce the same subshift as smaller tileset, so only 28 equivalence classes of tilesets need to be considered. Table \ref{table_even} shows a tileset for each one of these 28 equivalence classes. Table \ref{table_even_stupid} shows the 8 equivalence classes of non-minimal tilesets, with one tileset per class.

\begin{table}[!ht]
\centering
\begin{tabular}{|l|l|l|}
\hline
\rule{0mm}{2.5ex}In $\gen$&Not in $\gen$ (not in $\nar$?)&Not in $\nar$\\
\hline\hline
\rule{0mm}{3.2ex}\includegraphics{even_classes/small_class-1-0}&\includegraphics{even_classes/small_class-14-0}&\includegraphics{even_classes/small_class-19-4}\\
\includegraphics{even_classes/small_class-16-0}&\includegraphics{even_classes/small_class-23-0}&\includegraphics{even_classes/small_class-9-1}\\
\includegraphics{even_classes/small_class-4-0}&\includegraphics{even_classes/small_class-15-0}&\includegraphics{even_classes/small_class-20-6}\\
\includegraphics{even_classes/small_class-10-0}&\includegraphics{even_classes/small_class-27-0}&\includegraphics{even_classes/small_class-32-1}\\
\includegraphics{even_classes/small_class-17-0}&\includegraphics{even_classes/small_class-34-0}&\includegraphics{even_classes/small_class-21-8}\\
\includegraphics{even_classes/small_class-18-4}&\includegraphics{even_classes/small_class-25-0}&\includegraphics{even_classes/small_class-24-1}\\
\includegraphics{even_classes/small_class-11-0}&\includegraphics{even_classes/small_class-29-1}&\includegraphics{even_classes/small_class-33-2}\\
\includegraphics{even_classes/small_class-31-0}&\includegraphics{even_classes/small_class-35-0}&\\\cline{2-3}
\multicolumn{3}{|l|}{\includegraphics{even_classes/small_class-22-0}}\\
\multicolumn{3}{|l|}{\includegraphics{even_classes/small_class-26-0}}\\
\multicolumn{3}{|l|}{\includegraphics{even_classes/small_class-28-0}}\\
\multicolumn{3}{|l|}{\includegraphics{even_classes/small_class-30-0}}\\
\multicolumn{3}{|l|}{\includegraphics{even_classes/small_class-36-0}}\\
\hline
\end{tabular}
\caption{Classification of the even bicolor tilesets}\label{table_even}
\end{table}

For the subshifts that do not belong to~$\gen$, we do not know whether they belong to~$\nar$, although we expect that it is not the case.

\begin{table}[!ht]
\centering
\begin{tabular}{|l|l|}
\hline
\rule{0mm}{2.5ex}Empty subshift&Unused tile\\
\hline\hline
\rule{0mm}{3.2ex}\includegraphics{even_classes/small_class-2-1}&\includegraphics{even_classes/small_class-3-4}\\
\includegraphics{even_classes/small_class-6-0}&\includegraphics{even_classes/small_class-5-1}\\
&\includegraphics{even_classes/small_class-7-0}\\
&\includegraphics{even_classes/small_class-12-0}\\
&\includegraphics{even_classes/small_class-8-0}\\
&\includegraphics{even_classes/small_class-13-0}\\
\hline
\end{tabular}
\caption{The other even bicolor tilesets}\label{table_even_stupid}
\end{table}

In order to organize the classification, we will indicate the number of \emph{corner} tiles of a tileset, which are the tiles \smalltile{1100}, \smalltile{0110}, \smalltile{0011} and \smalltile{1001}. Note that this set of tiles is invariant under all the symmetries, so the number of corners in a tileset is invariant under the equivalence relation.

\subsection{Even tilesets of size 1}

There are 8 even tilesets of size 1, divided in 2 equivalence classes.
\subsubsection{No corner}\label{class_11}
In~$\gen$. There is one constant tiling.

\vspace{2mm}
\hspace{\stretch{1}}
\includegraphics{even_classes/small_class-1-0}
\hspace{\stretch{2}}
\includegraphics{even_classes/small_class-1-1}
\hspace{\stretch{2}}
\includegraphics{even_classes/small_class-1-2}
\hspace{\stretch{2}}
\includegraphics{even_classes/small_class-1-3}
\hspace{\stretch{1}}

\subsubsection{One corner}
There is no tiling.

\vspace{2mm}
\hspace{\stretch{1}}
\includegraphics{even_classes/small_class-2-0}
\hspace{\stretch{2}}
\includegraphics{even_classes/small_class-2-1}
\hspace{\stretch{2}}
\includegraphics{even_classes/small_class-2-2}
\hspace{\stretch{2}}
\includegraphics{even_classes/small_class-2-3}
\hspace{\stretch{1}}

\subsection{Even tilesets of size 2}
There are~$\binom{8}{2}=28$ even tilesets of size 2, divided in 5 equivalence classes.

\subsubsection{No corner (first case)}
In~$\gen$. There are 2 tilings. It is conjugate to~$\{0,1\}$ with the trivial action~$(i,j)\cdot x=x$.

\vspace{2mm}
\hspace{\stretch{1}}
\includegraphics{even_classes/small_class-16-0}
\hspace{\stretch{2}}
\includegraphics{even_classes/small_class-16-1}
\hspace{\stretch{1}}

\subsubsection{No corner (second case)}\label{class_22}
In~$\gen$. A configuration using the first tileset is a configuration of the~$\Z$-fullshift over~$\{0,1\}$, copied on each row. It is therefore conjugate to~$\{0,1\}^{\Z^2/(0,1)\Z}$.

\vspace{2mm}
\hspace{\stretch{1}}
\includegraphics{even_classes/small_class-4-2}
\hspace{\stretch{2}}
\includegraphics{even_classes/small_class-4-0}
\hspace{\stretch{2}}
\includegraphics{even_classes/small_class-4-1}
\hspace{\stretch{2}}
\includegraphics{even_classes/small_class-4-3}
\hspace{\stretch{1}}

\subsubsection{One corner + one even tile}
In~$\gen$. The corner cannot be used, the induced subshifts are the same as in Class \ref{class_11}.

\vspace{2mm}
\hspace{\stretch{1}}
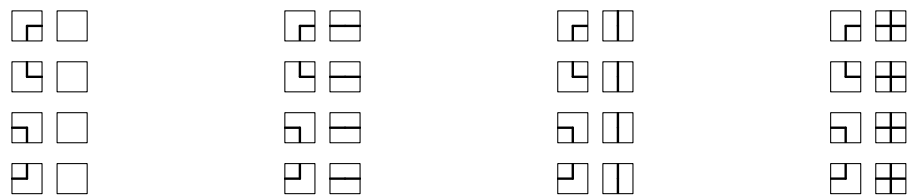
\includegraphics{even_classes/small_class-3-12}
\hspace{\stretch{2}}
\includegraphics{even_classes/small_class-3-9}
\hspace{\stretch{2}}
\includegraphics{even_classes/small_class-3-6}
\hspace{\stretch{2}}
\includegraphics{even_classes/small_class-3-3}
\hspace{\stretch{1}}

\vspace{2mm}
\hspace{\stretch{1}}
\includegraphics{even_classes/small_class-3-4}
\hspace{\stretch{2}}
\includegraphics{even_classes/small_class-3-1}
\hspace{\stretch{2}}
\includegraphics{even_classes/small_class-3-14}
\hspace{\stretch{2}}
\includegraphics{even_classes/small_class-3-11}
\hspace{\stretch{1}}

\vspace{2mm}
\hspace{\stretch{1}}
\includegraphics{even_classes/small_class-3-8}
\hspace{\stretch{2}}
\includegraphics{even_classes/small_class-3-13}
\hspace{\stretch{2}}
\includegraphics{even_classes/small_class-3-2}
\hspace{\stretch{2}}
\includegraphics{even_classes/small_class-3-7}
\hspace{\stretch{1}}

\vspace{2mm}
\hspace{\stretch{1}}
\includegraphics{even_classes/small_class-3-0}
\hspace{\stretch{2}}
\includegraphics{even_classes/small_class-3-5}
\hspace{\stretch{2}}
\includegraphics{even_classes/small_class-3-10}
\hspace{\stretch{2}}
\includegraphics{even_classes/small_class-3-15}
\hspace{\stretch{1}}

\subsubsection{Two opposite corners}\label{class_24}
In~$\gen$. There are 2 configurations, the subshift is conjugate to~$\Z^2/H$ where~$H=\{(x,y)\in\Z^2:x+y=0\mod 2\}$, equivalently to~$\{0,1\}$ with the action~$(i,j)\cdot a=a+i+j\mod 2$.

\vspace{2mm}
\hspace{\stretch{1}}
\includegraphics{even_classes/small_class-10-0}
\hspace{\stretch{2}}
\includegraphics{even_classes/small_class-10-1}
\hspace{\stretch{1}}

\subsubsection{Two adjacent corners}
There is no tiling.

\vspace{2mm}
\hspace{\stretch{1}}
\includegraphics{even_classes/small_class-6-0}
\hspace{\stretch{2}}
\includegraphics{even_classes/small_class-6-1}
\hspace{\stretch{2}}
\includegraphics{even_classes/small_class-6-2}
\hspace{\stretch{2}}
\includegraphics{even_classes/small_class-6-3}
\hspace{\stretch{1}}

\subsection{Even tilesets of size 3}
There are~$\binom{8}{3}=56$ even tilesets of size 3, divided in 6 equivalence classes.

\subsubsection{No corner}
In~$\gen$. The subshift is the union of two subshifts from Class \ref{class_22}.

\vspace{2mm}
\hspace{\stretch{1}}
\includegraphics{even_classes/small_class-17-0}
\hspace{\stretch{2}}
\includegraphics{even_classes/small_class-17-3}
\hspace{\stretch{2}}
\includegraphics{even_classes/small_class-17-1}
\hspace{\stretch{2}}
\includegraphics{even_classes/small_class-17-2}
\hspace{\stretch{1}}

\subsubsection{One corner + two even tiles (first case)}
In~$\gen$. The subshift is countable: it contains the two uniform configurations made of \smalltile{1010} or \smalltile{0101} and each other configuration is completely determined by one position of the tile \smalltile{1100}.

\vspace{2mm}
\hspace{\stretch{1}}
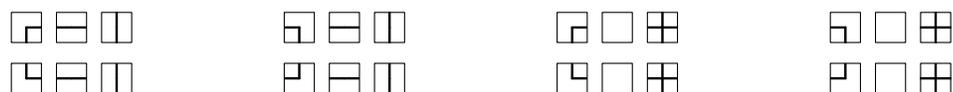
\includegraphics{even_classes/small_class-18-3}
\hspace{\stretch{2}}
\includegraphics{even_classes/small_class-18-7}
\hspace{\stretch{2}}
\includegraphics{even_classes/small_class-18-6}
\hspace{\stretch{2}}
\includegraphics{even_classes/small_class-18-2}
\hspace{\stretch{1}}

\vspace{2mm}
\hspace{\stretch{1}}
\includegraphics{even_classes/small_class-18-4}
\hspace{\stretch{2}}
\includegraphics{even_classes/small_class-18-0}
\hspace{\stretch{2}}
\includegraphics{even_classes/small_class-18-1}
\hspace{\stretch{2}}
\includegraphics{even_classes/small_class-18-5}
\hspace{\stretch{1}}

\subsubsection{One corner + two even tiles (second case)}
In~$\gen$. The corner cannot be used, the induced subshifts are the same as in Class \ref{class_22}.

\vspace{2mm}
\hspace{\stretch{1}}
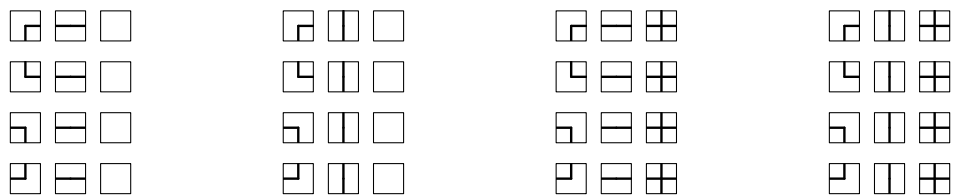
\includegraphics{even_classes/small_class-5-5}
\hspace{\stretch{2}}
\includegraphics{even_classes/small_class-5-8}
\hspace{\stretch{2}}
\includegraphics{even_classes/small_class-5-15}
\hspace{\stretch{2}}
\includegraphics{even_classes/small_class-5-3}
\hspace{\stretch{1}}

\vspace{2mm}
\hspace{\stretch{1}}
\includegraphics{even_classes/small_class-5-1}
\hspace{\stretch{2}}
\includegraphics{even_classes/small_class-5-10}
\hspace{\stretch{2}}
\includegraphics{even_classes/small_class-5-13}
\hspace{\stretch{2}}
\includegraphics{even_classes/small_class-5-7}
\hspace{\stretch{1}}

\vspace{2mm}
\hspace{\stretch{1}}
\includegraphics{even_classes/small_class-5-4}
\hspace{\stretch{2}}
\includegraphics{even_classes/small_class-5-14}
\hspace{\stretch{2}}
\includegraphics{even_classes/small_class-5-9}
\hspace{\stretch{2}}
\includegraphics{even_classes/small_class-5-2}
\hspace{\stretch{1}}

\vspace{2mm}
\hspace{\stretch{1}}
\includegraphics{even_classes/small_class-5-0}
\hspace{\stretch{2}}
\includegraphics{even_classes/small_class-5-12}
\hspace{\stretch{2}}
\includegraphics{even_classes/small_class-5-11}
\hspace{\stretch{2}}
\includegraphics{even_classes/small_class-5-6}
\hspace{\stretch{1}}

\subsubsection{Two adjacent corners + one even tile}
In~$\gen$. The corners cannot be used, the induced subshifts are the same as in Class \ref{class_11}.

\vspace{2mm}
\hspace{\stretch{1}}
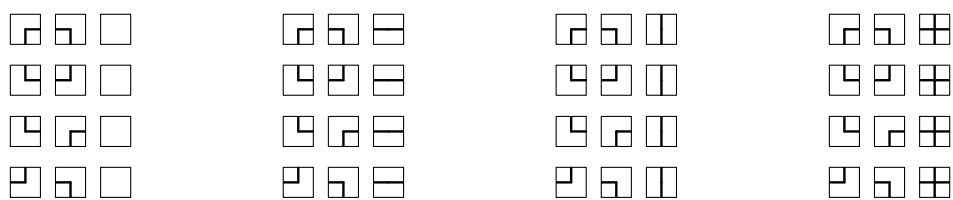
\includegraphics{even_classes/small_class-7-4}
\hspace{\stretch{2}}
\includegraphics{even_classes/small_class-7-5}
\hspace{\stretch{2}}
\includegraphics{even_classes/small_class-7-2}
\hspace{\stretch{2}}
\includegraphics{even_classes/small_class-7-3}
\hspace{\stretch{1}}

\vspace{2mm}
\hspace{\stretch{1}}
\includegraphics{even_classes/small_class-7-0}
\hspace{\stretch{2}}
\includegraphics{even_classes/small_class-7-1}
\hspace{\stretch{2}}
\includegraphics{even_classes/small_class-7-6}
\hspace{\stretch{2}}
\includegraphics{even_classes/small_class-7-7}
\hspace{\stretch{1}}

\vspace{2mm}
\hspace{\stretch{1}}
\includegraphics{even_classes/small_class-7-8}
\hspace{\stretch{2}}
\includegraphics{even_classes/small_class-7-13}
\hspace{\stretch{2}}
\includegraphics{even_classes/small_class-7-10}
\hspace{\stretch{2}}
\includegraphics{even_classes/small_class-7-15}
\hspace{\stretch{1}}

\vspace{2mm}
\hspace{\stretch{1}}
\includegraphics{even_classes/small_class-7-12}
\hspace{\stretch{2}}
\includegraphics{even_classes/small_class-7-9}
\hspace{\stretch{2}}
\includegraphics{even_classes/small_class-7-14}
\hspace{\stretch{2}}
\includegraphics{even_classes/small_class-7-11}
\hspace{\stretch{1}}

\subsubsection{Two opposite corners + one even tile}\label{class_35}
In~$\gen$. The configurations using the first tileset are~$(1,-1)$-periodic, and the subshift is a factor of the~$(1,-1)$-periodic shift~$P$ over~$\{0,1\}$. The map~$f:P\to X$ works in two steps. First, the function~$f_1:P\to P$ removes the pattern~$11$. It is defined by~$f_1(x)_{i,j}=1\iff x_{i-1,j}x_{i,j}=01$. $f_1$ is a factor map and~$\im{f_1}$ is the set of~$(1,-1)$-periodic configurations avoiding the pattern~$11$. Next,~$f_2:\im{f_1}\to X$ assigns color~$x_{i,j}$ to the left edge of cell~$(i,j)$, which uniquely determines the configuration. The composition~$f=f_2\circ f_1$ is the sought factor map.


\vspace{2mm}
\hspace{\stretch{1}}
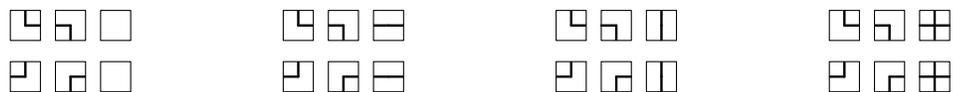
\includegraphics{even_classes/small_class-11-0}
\hspace{\stretch{2}}
\includegraphics{even_classes/small_class-11-5}
\hspace{\stretch{2}}
\includegraphics{even_classes/small_class-11-6}
\hspace{\stretch{2}}
\includegraphics{even_classes/small_class-11-3}
\hspace{\stretch{1}}

\vspace{2mm}
\hspace{\stretch{1}}
\includegraphics{even_classes/small_class-11-4}
\hspace{\stretch{2}}
\includegraphics{even_classes/small_class-11-1}
\hspace{\stretch{2}}
\includegraphics{even_classes/small_class-11-2}
\hspace{\stretch{2}}
\includegraphics{even_classes/small_class-11-7}
\hspace{\stretch{1}}

\subsubsection{Three corners}
In~$\gen$. Only the two opposite corners can be used, the induced subshifts are the same as in Class \ref{class_24}.

\vspace{2mm}
\hspace{\stretch{1}}
\includegraphics{even_classes/small_class-12-0}
\hspace{\stretch{2}}
\includegraphics{even_classes/small_class-12-1}
\hspace{\stretch{2}}
\includegraphics{even_classes/small_class-12-2}
\hspace{\stretch{2}}
\includegraphics{even_classes/small_class-12-3}
\hspace{\stretch{1}}

\subsection{Even tilesets of size 4}
There are~$\binom{8}{4}=70$ even tilesets of size 4, divided in 9 equivalence classes.

\subsubsection{No corner}
In~$\gen$. This tileset is the set of all tiles whose opposite edges have the same color. There is a conjugacy~$f:\{0,1\}^{\Z^2/(1,0)\Z}\times \{0,1\}^{\Z^2/(0,1)\Z}\to X$. On input~$(x,y)$, each cell~$(i,j)$ is assigned color~$x_{i,j}$ to its left and right edge and~$y_{i,j}$ to its lower and upper edge. Note that~$x_{i,j}$ only depends on~$j$ and~$y_{i,j}$ only depends on~$i$.

\vspace{2mm}
\hspace{\stretch{1}}
\includegraphics{even_classes/small_class-31-0}
\hspace{\stretch{1}}

\subsubsection{One corner + three even tiles}\label{class_irr_chevron}
Not in~$\nar$ (Theorem \ref{thm_tilesets_nar})

\vspace{2mm}
\hspace{\stretch{1}}
\includegraphics{even_classes/small_class-19-12}
\hspace{\stretch{2}}
\includegraphics{even_classes/small_class-19-9}
\hspace{\stretch{2}}
\includegraphics{even_classes/small_class-19-6}
\hspace{\stretch{2}}
\includegraphics{even_classes/small_class-19-3}
\hspace{\stretch{1}}

\vspace{2mm}
\hspace{\stretch{1}}
\includegraphics{even_classes/small_class-19-4}
\hspace{\stretch{2}}
\includegraphics{even_classes/small_class-19-1}
\hspace{\stretch{2}}
\includegraphics{even_classes/small_class-19-14}
\hspace{\stretch{2}}
\includegraphics{even_classes/small_class-19-11}
\hspace{\stretch{1}}

\vspace{2mm}
\hspace{\stretch{1}}
\includegraphics{even_classes/small_class-19-8}
\hspace{\stretch{2}}
\includegraphics{even_classes/small_class-19-13}
\hspace{\stretch{2}}
\includegraphics{even_classes/small_class-19-2}
\hspace{\stretch{2}}
\includegraphics{even_classes/small_class-19-7}
\hspace{\stretch{1}}

\vspace{2mm}
\hspace{\stretch{1}}
\includegraphics{even_classes/small_class-19-0}
\hspace{\stretch{2}}
\includegraphics{even_classes/small_class-19-5}
\hspace{\stretch{2}}
\includegraphics{even_classes/small_class-19-10}
\hspace{\stretch{2}}
\includegraphics{even_classes/small_class-19-15}
\hspace{\stretch{1}}

\subsubsection{Two adjacents corners + two even tiles (first case)}\label{class_nested_staples}
Not in~$\nar$ (Theorem \ref{thm_tilesets_nar}).

\vspace{2mm}
\hspace{\stretch{1}}
\includegraphics{even_classes/small_class-20-0}
\hspace{\stretch{2}}
\includegraphics{even_classes/small_class-20-1}
\hspace{\stretch{2}}
\includegraphics{even_classes/small_class-20-2}
\hspace{\stretch{2}}
\includegraphics{even_classes/small_class-20-3}
\hspace{\stretch{1}}

\vspace{2mm}
\hspace{\stretch{1}}
\includegraphics{even_classes/small_class-20-4}
\hspace{\stretch{2}}
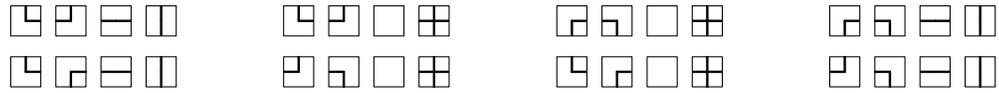
\includegraphics{even_classes/small_class-20-5}
\hspace{\stretch{2}}
\includegraphics{even_classes/small_class-20-6}
\hspace{\stretch{2}}
\includegraphics{even_classes/small_class-20-7}
\hspace{\stretch{1}}

\subsubsection{Two adjacent corners + two even tiles (second case)}\label{class_histograms}
Not in~$\nar$ (Theorem \ref{thm_tilesets_nar}).

\vspace{2mm}
\hspace{\stretch{1}}
\includegraphics{even_classes/small_class-9-1}
\hspace{\stretch{2}}
\includegraphics{even_classes/small_class-9-4}
\hspace{\stretch{2}}
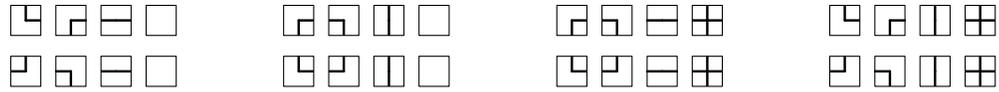
\includegraphics{even_classes/small_class-9-5}
\hspace{\stretch{2}}
\includegraphics{even_classes/small_class-9-3}
\hspace{\stretch{1}}

\vspace{2mm}
\hspace{\stretch{1}}
\includegraphics{even_classes/small_class-9-0}
\hspace{\stretch{2}}
\includegraphics{even_classes/small_class-9-6}
\hspace{\stretch{2}}
\includegraphics{even_classes/small_class-9-7}
\hspace{\stretch{2}}
\includegraphics{even_classes/small_class-9-2}
\hspace{\stretch{1}}

\subsubsection{Two adjacent corners + two even tiles (third case)}
In~$\gen$. The corners cannot be used, the induced subshifts are the same as in Class \ref{class_22}.

\vspace{2mm}
\hspace{\stretch{1}}
\includegraphics{even_classes/small_class-8-2}
\hspace{\stretch{2}}
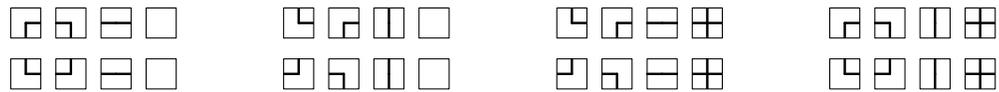
\includegraphics{even_classes/small_class-8-4}
\hspace{\stretch{2}}
\includegraphics{even_classes/small_class-8-7}
\hspace{\stretch{2}}
\includegraphics{even_classes/small_class-8-1}
\hspace{\stretch{1}}

\vspace{2mm}
\hspace{\stretch{1}}
\includegraphics{even_classes/small_class-8-0}
\hspace{\stretch{2}}
\includegraphics{even_classes/small_class-8-6}
\hspace{\stretch{2}}
\includegraphics{even_classes/small_class-8-5}
\hspace{\stretch{2}}
\includegraphics{even_classes/small_class-8-3}
\hspace{\stretch{1}}

\subsubsection{Two opposite corners + two even tiles (first case)}
In~$\gen$. Let~$X$ be induced by the first tileset. Its configurations are~$(1,1)$-periodic and it is a factor of the~$(1,1)$-periodic shift on~$\{0,1\}$. The factor map~$f:\{0,1\}^{\Z^2/(1,1)\Z}\to X$ on input~$x$ assigns color~$x_{i,j}$ to the left edge of cell~$(i,j)$, which uniquely determines the tiles.


\vspace{2mm}
\hspace{\stretch{1}}
\includegraphics{even_classes/small_class-22-0}
\hspace{\stretch{2}}
\includegraphics{even_classes/small_class-22-2}
\hspace{\stretch{2}}
\includegraphics{even_classes/small_class-22-1}
\hspace{\stretch{2}}
\includegraphics{even_classes/small_class-22-3}
\hspace{\stretch{1}}

\subsubsection{Two opposite corners + two even tiles (second case)}\label{class_stretched_irr_stairs}
Not in $\gen$ (Theorem \ref{thm_wires_notC0}).

\vspace{2mm}
\hspace{\stretch{1}}
\includegraphics{even_classes/small_class-14-1}
\hspace{\stretch{2}}
\includegraphics{even_classes/small_class-14-6}
\hspace{\stretch{2}}
\includegraphics{even_classes/small_class-14-5}
\hspace{\stretch{2}}
\includegraphics{even_classes/small_class-14-2}
\hspace{\stretch{1}}

\vspace{2mm}
\hspace{\stretch{1}}
\includegraphics{even_classes/small_class-14-0}
\hspace{\stretch{2}}
\includegraphics{even_classes/small_class-14-4}
\hspace{\stretch{2}}
\includegraphics{even_classes/small_class-14-7}
\hspace{\stretch{2}}
\includegraphics{even_classes/small_class-14-3}
\hspace{\stretch{1}}

\subsubsection{Three corners + one even tile}

In~$\gen$. Only the two opposite corners can be used, the induced subshifts are the same as in Class \ref{class_35}.

\vspace{2mm}
\hspace{\stretch{1}}
\includegraphics{even_classes/small_class-13-8}
\hspace{\stretch{2}}
\includegraphics{even_classes/small_class-13-13}
\hspace{\stretch{2}}
\includegraphics{even_classes/small_class-13-2}
\hspace{\stretch{2}}
\includegraphics{even_classes/small_class-13-7}
\hspace{\stretch{1}}

\vspace{2mm}
\hspace{\stretch{1}}
\includegraphics{even_classes/small_class-13-0}
\hspace{\stretch{2}}
\includegraphics{even_classes/small_class-13-5}
\hspace{\stretch{2}}
\includegraphics{even_classes/small_class-13-10}
\hspace{\stretch{2}}
\includegraphics{even_classes/small_class-13-15}
\hspace{\stretch{1}}

\vspace{2mm}
\hspace{\stretch{1}}
\includegraphics{even_classes/small_class-13-12}
\hspace{\stretch{2}}
\includegraphics{even_classes/small_class-13-9}
\hspace{\stretch{2}}
\includegraphics{even_classes/small_class-13-6}
\hspace{\stretch{2}}
\includegraphics{even_classes/small_class-13-3}
\hspace{\stretch{1}}

\vspace{2mm}
\hspace{\stretch{1}}
\includegraphics{even_classes/small_class-13-4}
\hspace{\stretch{2}}
\includegraphics{even_classes/small_class-13-1}
\hspace{\stretch{2}}
\includegraphics{even_classes/small_class-13-14}
\hspace{\stretch{2}}
\includegraphics{even_classes/small_class-13-11}
\hspace{\stretch{1}}

\subsubsection{Four corners}\label{class_corners}
In~$\gen$ (Section \myref{sec_corners} in \cite{FH24}).

\vspace{2mm}
\hspace{\stretch{1}}
\includegraphics{even_classes/small_class-26-0}
\hspace{\stretch{1}}

\subsection{Even tilesets of size 5}
There are~$\binom{8}{5}=56$ even tilesets of size 5, divided in 6 equivalence classes.

\subsubsection{One corner + four even tiles}\label{class_crossing_brackets}

Not in~$\nar$ (Theorem \ref{thm_tilesets_nar}).

\vspace{2mm}
\hspace{\stretch{1}}
\includegraphics{even_classes/small_class-32-3}
\hspace{\stretch{2}}
\includegraphics{even_classes/small_class-32-2}
\hspace{\stretch{1}}

\vspace{2mm}
\hspace{\stretch{1}}
\includegraphics{even_classes/small_class-32-1}
\hspace{\stretch{2}}
\includegraphics{even_classes/small_class-32-0}
\hspace{\stretch{1}}

\subsubsection{Two adjacent corners + three even tiles}\label{class_staples_gaps}
Not in~$\nar$ (Theorem \ref{thm_tilesets_nar}).

\vspace{2mm}
\hspace{\stretch{1}}
\includegraphics{even_classes/small_class-21-4}
\hspace{\stretch{2}}
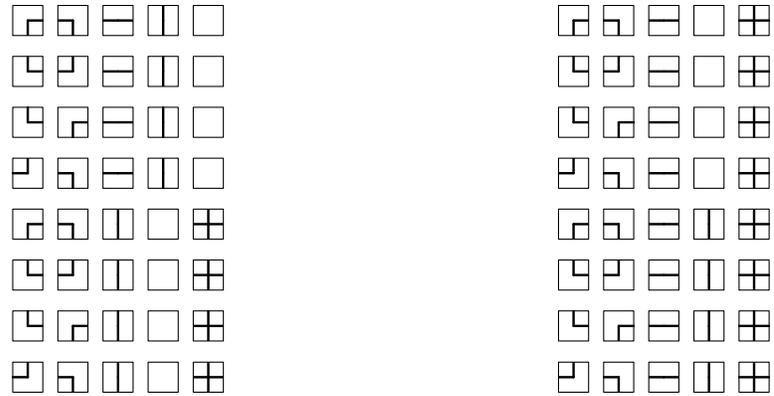
\includegraphics{even_classes/small_class-21-5}
\hspace{\stretch{1}}

\vspace{2mm}
\hspace{\stretch{1}}
\includegraphics{even_classes/small_class-21-0}
\hspace{\stretch{2}}
\includegraphics{even_classes/small_class-21-1}
\hspace{\stretch{1}}

\vspace{2mm}
\hspace{\stretch{1}}
\includegraphics{even_classes/small_class-21-8}
\hspace{\stretch{2}}
\includegraphics{even_classes/small_class-21-13}
\hspace{\stretch{1}}

\vspace{2mm}
\hspace{\stretch{1}}
\includegraphics{even_classes/small_class-21-12}
\hspace{\stretch{2}}
\includegraphics{even_classes/small_class-21-9}
\hspace{\stretch{1}}

\vspace{2mm}
\hspace{\stretch{1}}
\includegraphics{even_classes/small_class-21-2}
\hspace{\stretch{2}}
\includegraphics{even_classes/small_class-21-3}
\hspace{\stretch{1}}

\vspace{2mm}
\hspace{\stretch{1}}
\includegraphics{even_classes/small_class-21-6}
\hspace{\stretch{2}}
\includegraphics{even_classes/small_class-21-7}
\hspace{\stretch{1}}

\vspace{2mm}
\hspace{\stretch{1}}
\includegraphics{even_classes/small_class-21-10}
\hspace{\stretch{2}}
\includegraphics{even_classes/small_class-21-15}
\hspace{\stretch{1}}

\vspace{2mm}
\hspace{\stretch{1}}
\includegraphics{even_classes/small_class-21-14}
\hspace{\stretch{2}}
\includegraphics{even_classes/small_class-21-11}
\hspace{\stretch{1}}

\subsubsection{Two opposite corners + three even tiles}\label{class_irr_stairs}
Not in~$\gen$ (Theorem \ref{thm_wires_notC0}).

\vspace{2mm}
\hspace{\stretch{1}}
\includegraphics{even_classes/small_class-23-4}
\hspace{\stretch{2}}
\includegraphics{even_classes/small_class-23-1}
\hspace{\stretch{1}}

\vspace{2mm}
\hspace{\stretch{1}}
\includegraphics{even_classes/small_class-23-0}
\hspace{\stretch{2}}
\includegraphics{even_classes/small_class-23-5}
\hspace{\stretch{1}}

\vspace{2mm}
\hspace{\stretch{1}}
\includegraphics{even_classes/small_class-23-2}
\hspace{\stretch{2}}
\includegraphics{even_classes/small_class-23-7}
\hspace{\stretch{1}}

\vspace{2mm}
\hspace{\stretch{1}}
\includegraphics{even_classes/small_class-23-6}
\hspace{\stretch{2}}
\includegraphics{even_classes/small_class-23-3}
\hspace{\stretch{1}}

\subsubsection{Three corners + two even tiles (first case)}\label{class_slanted_histograms}
Not in~$\nar$ (Theorem \ref{thm_tilesets_nar}).

\vspace{2mm}
\hspace{\stretch{1}}
\includegraphics{even_classes/small_class-24-7}
\hspace{\stretch{2}}
\includegraphics{even_classes/small_class-24-2}
\hspace{\stretch{1}}

\vspace{2mm}
\hspace{\stretch{1}}
\includegraphics{even_classes/small_class-24-0}
\hspace{\stretch{2}}
\includegraphics{even_classes/small_class-24-5}
\hspace{\stretch{1}}

\vspace{2mm}
\hspace{\stretch{1}}
\includegraphics{even_classes/small_class-24-3}
\hspace{\stretch{2}}
\includegraphics{even_classes/small_class-24-6}
\hspace{\stretch{1}}

\vspace{2mm}
\hspace{\stretch{1}}
\includegraphics{even_classes/small_class-24-4}
\hspace{\stretch{2}}
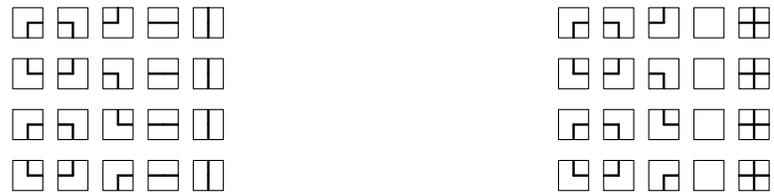
\includegraphics{even_classes/small_class-24-1}
\hspace{\stretch{1}}

\subsubsection{Three corners + two even tiles (second case)}\label{class_flames}
Not in $\gen$ (Theorem \ref{thm_wires_notC0}).

\vspace{2mm}
\hspace{\stretch{1}}
\includegraphics{even_classes/small_class-15-4}
\hspace{\stretch{2}}
\includegraphics{even_classes/small_class-15-2}
\hspace{\stretch{1}}

\vspace{2mm}
\hspace{\stretch{1}}
\includegraphics{even_classes/small_class-15-0}
\hspace{\stretch{2}}
\includegraphics{even_classes/small_class-15-6}
\hspace{\stretch{1}}

\vspace{2mm}
\hspace{\stretch{1}}
\includegraphics{even_classes/small_class-15-5}
\hspace{\stretch{2}}
\includegraphics{even_classes/small_class-15-3}
\hspace{\stretch{1}}

\vspace{2mm}
\hspace{\stretch{1}}
\includegraphics{even_classes/small_class-15-1}
\hspace{\stretch{2}}
\includegraphics{even_classes/small_class-15-7}
\hspace{\stretch{1}}

\vspace{2mm}
\hspace{\stretch{1}}
\includegraphics{even_classes/small_class-15-14}
\hspace{\stretch{2}}
\includegraphics{even_classes/small_class-15-9}
\hspace{\stretch{1}}

\vspace{2mm}
\hspace{\stretch{1}}
\includegraphics{even_classes/small_class-15-12}
\hspace{\stretch{2}}
\includegraphics{even_classes/small_class-15-11}
\hspace{\stretch{1}}

\vspace{2mm}
\hspace{\stretch{1}}
\includegraphics{even_classes/small_class-15-8}
\hspace{\stretch{2}}
\includegraphics{even_classes/small_class-15-15}
\hspace{\stretch{1}}

\vspace{2mm}
\hspace{\stretch{1}}
\includegraphics{even_classes/small_class-15-10}
\hspace{\stretch{2}}
\includegraphics{even_classes/small_class-15-13}
\hspace{\stretch{1}}

\subsubsection{Four corners + one even tile}\label{class_corners_white}
Not in~$\gen$ (Theorem \ref{thm_wires_notC0}).

\vspace{2mm}
\hspace{\stretch{1}}
\includegraphics{even_classes/small_class-27-0}
\hspace{\stretch{2}}
\includegraphics{even_classes/small_class-27-1}
\hspace{\stretch{1}}

\vspace{2mm}
\hspace{\stretch{1}}
\includegraphics{even_classes/small_class-27-2}
\hspace{\stretch{2}}
\includegraphics{even_classes/small_class-27-3}
\hspace{\stretch{1}}

\subsection{Even tilesets of size 6}
There are~$\binom{8}{6}=28$ even tilesets of size 6, divided in 5 equivalence classes.

\subsubsection{Two adjacent corners + four even tiles}\label{class_overlapping_staples}
Not in~$\nar$ (Theorem \ref{thm_tilesets_nar}).

\vspace{2mm}
\hspace{\stretch{1}}
\includegraphics{even_classes/small_class-33-0}
\hspace{\stretch{2}}
\includegraphics{even_classes/small_class-33-1}
\hspace{\stretch{1}}

\vspace{2mm}
\hspace{\stretch{1}}
\includegraphics{even_classes/small_class-33-2}
\hspace{\stretch{2}}
\includegraphics{even_classes/small_class-33-3}
\hspace{\stretch{1}}

\subsubsection{Two opposite corners + four even tiles}\label{class34}
Not in~$\gen$ (Theorem \ref{thm_wires_notC0}).

\vspace{2mm}
\hspace{\stretch{1}}
\includegraphics{even_classes/small_class-34-0}
\hspace{\stretch{2}}
\includegraphics{even_classes/small_class-34-1}
\hspace{\stretch{1}}

\subsubsection{Three corners + three even tiles}\label{class25}
Not in~$\gen$ (Theorem \ref{thm_wires_notC0}).

\vspace{2mm}
\hspace{\stretch{1}}
\includegraphics{even_classes/small_class-25-8}
\hspace{\stretch{2}}
\includegraphics{even_classes/small_class-25-13}
\hspace{\stretch{1}}

\vspace{2mm}
\hspace{\stretch{1}}
\includegraphics{even_classes/small_class-25-0}
\hspace{\stretch{2}}
\includegraphics{even_classes/small_class-25-5}
\hspace{\stretch{1}}

\vspace{2mm}
\hspace{\stretch{1}}
\includegraphics{even_classes/small_class-25-12}
\hspace{\stretch{2}}
\includegraphics{even_classes/small_class-25-9}
\hspace{\stretch{1}}

\vspace{2mm}
\hspace{\stretch{1}}
\includegraphics{even_classes/small_class-25-4}
\hspace{\stretch{2}}
\includegraphics{even_classes/small_class-25-1}
\hspace{\stretch{1}}

\vspace{2mm}
\hspace{\stretch{1}}
\includegraphics{even_classes/small_class-25-2}
\hspace{\stretch{2}}
\includegraphics{even_classes/small_class-25-7}
\hspace{\stretch{1}}

\vspace{2mm}
\hspace{\stretch{1}}
\includegraphics{even_classes/small_class-25-10}
\hspace{\stretch{2}}
\includegraphics{even_classes/small_class-25-15}
\hspace{\stretch{1}}

\vspace{2mm}
\hspace{\stretch{1}}
\includegraphics{even_classes/small_class-25-6}
\hspace{\stretch{2}}
\includegraphics{even_classes/small_class-25-3}
\hspace{\stretch{1}}

\vspace{2mm}
\hspace{\stretch{1}}
\includegraphics{even_classes/small_class-25-14}
\hspace{\stretch{2}}
\includegraphics{even_classes/small_class-25-11}
\hspace{\stretch{1}}

\subsubsection{Four corners + two even tiles (first case)}\label{class_wires}
Not in~$\gen$ (Theorem \ref{thm_wires_notC0}).

\vspace{2mm}
\hspace{\stretch{1}}
\includegraphics{even_classes/small_class-29-0}
\hspace{\stretch{2}}
\includegraphics{even_classes/small_class-29-1}
\hspace{\stretch{1}}

\subsubsection{Four corners + two even tiles (second case)}\label{class_corners_2_even_II}
In~$\gen$ (Section \ref{sec_corners_3_even}).

\vspace{2mm}
\hspace{\stretch{1}}
\includegraphics{even_classes/small_class-28-0}
\hspace{\stretch{2}}
\includegraphics{even_classes/small_class-28-1}
\hspace{\stretch{1}}

\vspace{2mm}
\hspace{\stretch{1}}
\includegraphics{even_classes/small_class-28-2}
\hspace{\stretch{2}}
\includegraphics{even_classes/small_class-28-3}
\hspace{\stretch{1}}

\subsection{Even tilesets of size 7}
There are~$\binom{8}{7}=8$ even tilesets of size 7, divided in 2 equivalence classes.

\subsubsection{Three corners + four even tiles}\label{class35}
Not in~$\gen$ (Theorem \ref{thm_wires_notC0}).

\vspace{2mm}
\hspace{\stretch{1}}
\includegraphics{even_classes/small_class-35-2}
\hspace{\stretch{2}}
\includegraphics{even_classes/small_class-35-3}
\hspace{\stretch{1}}

\vspace{2mm}
\hspace{\stretch{1}}
\includegraphics{even_classes/small_class-35-0}
\hspace{\stretch{2}}
\includegraphics{even_classes/small_class-35-1}
\hspace{\stretch{1}}

\subsubsection{Four corners + three even tiles}\label{class_corners_3_even}
In~$\gen$ (Section \ref{sec_corners_3_even}).

\vspace{2mm}
\hspace{\stretch{1}}
\includegraphics{even_classes/small_class-30-0}
\hspace{\stretch{2}}
\includegraphics{even_classes/small_class-30-1}
\hspace{\stretch{1}}

\vspace{2mm}
\hspace{\stretch{1}}
\includegraphics{even_classes/small_class-30-2}
\hspace{\stretch{2}}
\includegraphics{even_classes/small_class-30-3}
\hspace{\stretch{1}}

\subsection{Even tileset of size 8}\label{class_all_even}
There is 1 even tileset of size 8.

\subsubsection{All the even tiles}

In~$\gen$ (Section \myref{sec_all_even} in \cite{FH24}).

\vspace{2mm}
\hspace{\stretch{1}}
\includegraphics{even_classes/small_class-36-0}
\hspace{\stretch{1}}

\section*{Acknowledgments}
\addcontentsline{toc}{section}{Acknowledgements}

We thank Guilhem Gamard, Emmanuel Jeandel, Julien Provillard and Alexis Terrassin for interesting discussions on the subject.


\end{document}